\documentclass{amsart}

\usepackage{amsmath,amsfonts,amsthm,amssymb,verbatim,enumerate,graphicx,mathtools}
\usepackage{enumitem}
\usepackage{color}
\usepackage[flushmargin]{footmisc}
\newtheorem{theorem}{Theorem}[section]
\newtheorem{lemma}[theorem]{Lemma}
\newtheorem{proposition}[theorem]{Proposition}
\newtheorem{corollary}[theorem]{Corollary}
\newtheorem{definition}{Definition}[section]

\newtheorem*{remark}{Remark}
\newtheorem*{remarks}{Remarks}
\numberwithin{equation}{section}
\newcommand{\tef}{transcendental entire function}
\newcommand\qfor{\quad\text{for }}
\newcommand \C{\mathbb{C}}
\newcommand \N{\mathbb{N}}
\newcommand \No{{\mathbb{N}_0}}
\newcommand \R{\mathbb{R}}
\newcommand \Z{\mathbb{Z}}
\makeatletter
\def\blfootnote{\xdef\@thefnmark{}\@footnotetext}
\makeatother
\begin{document}
%
%
%
%
\title[The dynamics of quasiregular maps of punctured space]{The dynamics of quasiregular maps of punctured space}
\author{Daniel A. Nicks}
\address{School of Mathematical Sciences \\ University of Nottingham \\ Nottingham
NG7 2RD \\ UK  \\  ORCiD: 0000-0002-9493-2970}
\email{dan.nicks@nottingham.ac.uk}
\author{David J. Sixsmith}
\address{Dept. of Mathematical Sciences \\
University of Liverpool \\
Liverpool L69 7ZL\\
UK \\
ORCiD: 0000-0002-3543-6969}
\email{david.sixsmith@open.ac.uk}

%
%
%
%
\begin{abstract}
The Fatou-Julia iteration theory of rational and transcendental entire functions has recently been extended to quasiregular maps in more than two real dimensions. Our goal in this paper is similar; we extend the iteration theory of analytic self-maps of the \emph{punctured plane} to quasiregular self-maps of \emph{punctured space}.

We define the Julia set as the set of points for which the complement of the forward orbit of any neighbourhood of the point is a finite set. We show that the Julia set is non-empty, and shares many properties with the classical Julia set of an analytic function. These properties are stronger than those known to hold for the Julia set of a general quasiregular map of space.

We define the quasi-Fatou set as the complement of the Julia set, and generalise a result of Baker concerning the topological properties of the components of this set. A key tool in the proof of these results is a version of the fast escaping set. We generalise various results of Mart{\'{\i}}-Pete concerning this set, for example showing that the Julia set is equal to the boundary of the fast escaping set.
\end{abstract}
\maketitle
%
%
%
%
\blfootnote{Both authors were supported by Engineering and Physical Sciences Research Council grant EP/L019841/1.}
\section{Introduction}
\subsection{Background}
Most studies of complex dynamics have considered analytic maps of either $\C$ or $\widehat{\C} := \C \cup\{\infty\}$; we refer to reference works such as \cite{beardon, MR1216719} for more information on complex dynamics. Various authors have studied the dynamics of analytic maps of the plane that have an additional essential singularity which is also an omitted value. Without loss of generality the singularity can be taken to be at the origin. These maps are known as \emph{transcendental analytic self-maps of $\C^*$}, where $\C^* := \C \setminus \{0\}$ is the \emph{punctured plane}. This study started with R{\aa}dstr{\"o}m \cite{MR0056702}, who pointed out that such maps are necessarily of the form
$$
z \mapsto z^k \exp(g(z) + h(1/z)),
$$
where $k \in \Z$ and $g, h$ are entire functions. Following R{\aa}dstr{\"o}m many authors have contributed to this work; see, for example, \cite{MR1687848, MR1357963, MR1279126, MR955806, MR1120038, MR916928}. We observe that, because of Picard's theorem, there are no analytic self-maps of the complex plane equipped with more than one puncture.

In this paper our goal is, for the first time, to extend this study to quasiregular maps of punctured \emph{space}. We defer the full definition of a quasiregular map to Section~\ref{section:defs}; for now it is sufficient to note that these maps are the natural generalization to higher dimensions of analytic maps in the plane. We set the following definitions in place for the remainder of the paper. Fix the dimension $d \geq 2$. Fix also the number of finite punctures $\nu\in\N$; this is always in addition to a singularity at infinity. We stress that $\nu$ is always taken to be at least one; for more information on quasiregular dynamics in space without punctures we refer to, for example, \cite{MR3009101,MR3265283}. Finally we define a set of punctures. For convenience we fix $y_0 = \infty$, we let $y_1, y_2, \ldots, y_\nu$ be distinct points of $\R^d$, and then we let $S \subset \widehat{\R^d} := \R^d \cup\{\infty\}$ be given by $$S := \{y_0, y_1, y_2, \ldots, y_\nu\}.$$

We are interested in the dynamics of a quasiregular map $f : \widehat{\R^d}\setminus S \to \widehat{\R^d}\setminus S$, with the property that $S$ coincides with the set of essential singularities of $f$. In this situation we say that $f$ is \emph{of \mbox{$S$-transcendental} type}. Note that an \emph{essential singularity} is a point at which no limit of $f$ exists.

It follows from Picard's theorem and Sto\"{\i}low factorization that a quasiregular map from $\R^2$ to $\R^2$ can omit at most one point. Hence, if $d=2$ then we must have $\nu = 1$. However, if $d \geq 3$, then we can take $\nu$ to be arbitrarily large. We demonstrate this as follows. Let $g : \widehat{\R^d} \to \widehat{\R^d}$ be a quasiregular map of degree $\nu + 1$, such that $g^{-1}(\infty) = \{y_0, y_1, y_2, \ldots, y_\nu\}$. Let $F : \R^d \to \R^d\setminus\{y_1, y_2, \ldots, y_\nu\}$ be a quasiregular map with an essential singularity at infinity. (The existence of such a map was shown by Drasin and Pankka \cite{DrasinandPankka}.) Then
$$
  f : \R^d\setminus\{y_1, y_2, \ldots, y_\nu\} \to \R^d\setminus\{y_1, y_2, \ldots, y_\nu\} \quad\text{where}\quad f := F \circ g,
$$
has essential singularities at $\{y_0, y_1, y_2, \ldots, y_\nu\}$, as required. We note that the class of maps such as $F$ above is not small, since if $h : \R^d \to \R^d$ is any quasiregular map, then $F \circ h : \R^d \to \R^d\setminus\{y_1, y_2, \ldots, y_\nu\}$ is also a quasiregular map.
\subsection{The Julia set}
Our definition of the Julia set follows earlier definitions of the Julia set of a quasiregular map by using a version of the so-called \emph{blowing-up property}. 
Suppose that $f:\widehat{\R^d}\setminus S \to \widehat{\R^d} \setminus S$ is a quasiregular map of \mbox{$S$-transcendental} type, that $x \in \widehat{\R^d}$ and that $U \subset \widehat{\R^d}\setminus S$. We define the \emph{backward orbit} of $x$ by
$$
O^-(x) := \bigcup_{k \ge 0} f^{-k}(x),
$$
and we also define the \emph{forward orbit} of $U$ by
$$
O^+(U) := \bigcup_{k\in\N} f^{k}(U).
$$
\begin{definition}\normalfont\label{jdef}
The \emph{Julia set $J(f)$} is defined to be the set of all points $x \in \widehat{\R^d}\setminus S$ such that, for every neighbourhood $U$ of $x$, the set $\widehat{\R^d} \setminus O^+(U)$ is finite.
\end{definition}
It can then be deduced that the Julia set is closed in $\widehat{\R^d}\setminus S$, and is completely invariant. Here we say that a set $X\subset\widehat{\R^d}\setminus S$ is \emph{completely invariant} if $x \in X$ if and only if $f(x) \in X$.

We show that our definition of the Julia set is consistent with the classical definition used for transcendental analytic self-maps of the punctured plane. Recall that the classical definition first defines the \emph{Fatou set} as the set of points that have a neighbourhood in which the set of iterates is a normal family, and then defines the Julia set as the complement of the Fatou set.
\begin{theorem}
\label{theo:Jclassical}
Suppose that $f$ is a transcendental analytic self-map of the punctured plane. Then the classical definition of $J(f)$ agrees with Definition~\ref{jdef}.
\end{theorem}

We next define the \emph{exceptional set $E(f)$} as
$$
E(f) := \left\{ x \in \widehat{\R^d} : \operatorname{card }\left(O^-(x)\right) < \infty\right\}.
$$
Clearly $S \subset E(f)$. It is a consequence of a well-known result of Rickman \cite{MR583633} that $E(f)$ is a finite set; see Lemma~\ref{greatpicard} below. 

We now give our principal result regarding the Julia set of a quasiregular map of $S$-transcendental type. Note that here, and elsewhere in the paper, the topological operations of closure, complement and boundary are taken with respect to $\widehat{\R^d}\setminus S$ unless otherwise specified. Also, if a set $X \subset \widehat{\R^d} \setminus S$ is such that the closure of $X$ in $\widehat{\R^d}$ meets $S$, then we say that $X$ is \emph{$S$-unbounded}; otherwise we say that $X$ is \emph{$S$-bounded}.
\begin{theorem}
\label{theo:J}
Suppose that $f : \widehat{\R^d}\setminus S \to \widehat{\R^d} \setminus S$ is a quasiregular map of \mbox{$S$-transcendental} type. Then the following hold.
\begin{enumerate}[label=(\alph*)]
\item The Julia set of $f$ is non-empty and perfect.\label{jperfect}
\item For each $x \in J(f) \setminus E(f)$, we have $J(f) = \overline{O^-(x)}$.\label{jback}
\item We have $J(f) = J(f^p)$, for $p\in\N$.\label{jnice}
\item Either $J(f)$ is connected or $J(f)$ has infinitely many components.\label{jconn}
\item All components of $J(f)$ are $S$-unbounded.\label{jnoSbounded}
\end{enumerate}
\end{theorem}

Clearly, it follows from \ref{jperfect} and \ref{jnoSbounded} that $J(f)$ has a connected component that contains at least two points. We can deduce that the Julia set has Hausdorff dimension at least one.

\begin{remarks}\normalfont\mbox{}

\begin{enumerate}
\item It follows from Theorem~\ref{theo:Jclassical} that Theorem~\ref{theo:J}\ref{jconn} and \ref{jnoSbounded} are generalisations to quasiregular maps of Baker and Dom{\'{\i}}nguez's results \cite[Theorem 3 and Theorem 2]{MR1687848}, which concern transcendental analytic self-maps of $\C^*$.

\item The dynamics of general quasiregular self-maps of $\R^d$ and $\widehat{\R^d}$ is the subject of \cite{MR3009101,MR3265283}, where a slightly weaker definition of the Julia set is adopted. In this definition the complement of the forward orbit of any neighbourhood is constrained only to be ``small''; see Definition~\ref{jcapdef} in Section~\ref{section:julia} below. In the cases studied in \cite{MR3009101,MR3265283}, the Julia set is generally non-empty, but otherwise properties analogous to Theorem~\ref{theo:J}\ref{jperfect}--\ref{jnice} are only known to hold under additional hypotheses such as Lipschitz continuity. Even with this extra assumption, it is only  known that the Hausdorff dimension of the Julia set is positive.

\item The dynamics of \emph{local uniformly} quasiregular maps of punctured \emph{manifolds} was studied by Okuyama and Pankka in \cite{MR3220456}.
\end{enumerate}
\end{remarks}
%
%
%
\subsection{The quasi-Fatou set}
Following \cite{SixsmithNicks1, SixsmithNicks2, SixsmithNicks3}, we define the \emph{quasi-Fatou set $QF(f)$} as the complement of the Julia set (recall that complements are taken in $\widehat{\R^d}\setminus S$). It is straightforward to show that the quasi-Fatou set is an open, completely invariant set which, if non-empty, has the Julia set as its boundary. We call the connected components of $QF(f)$ the \emph{quasi-Fatou components}.

Baker \cite[Theorem 1]{MR951969} (see also \cite[Theorem 1]{MR1687848} and \cite{MR916928}) showed that, for a transcendental analytic self-map of the punctured plane, all the components of the Fatou set are simply-connected, apart from at most one, which, if it exists, must be doubly-connected. In view of Theorem~\ref{theo:Jclassical}, our next result is a generalisation of this fact. Here, if a set $X \subset \R^d\setminus S$ is such that all complementary components of $X$ are $S$-unbounded, then we say that $X$ is \emph{$S$-full}; otherwise we say that $X$ is \emph{$S$-hollow}.
\begin{theorem}
\label{theo:F}
Suppose that $f : \widehat{\R^d}\setminus S \to \widehat{\R^d} \setminus S$ is a quasiregular map of \mbox{$S$-transcendental} type. Then all components of $QF(f)$ are $S$-full. Moreover there are at most $\nu$ components of $QF(f)$ which have more than one complementary component in $\widehat{\R^d}$.
\end{theorem}

Suppose that $d=2$, that $S = \{0, \infty\}$ (in which case $\nu=1$), and that $U \subset \C^*$ is a domain that is $S$-full. Set $W = \widehat{\C} \setminus U$. Then either $W$ is connected, in which case $U$ is simply-connected, or $W$ has two components (one containing $\infty$ and one containing $0$), in which case $U$ is doubly-connected. It follows that Theorem~\ref{theo:F} is indeed a generalisation of \cite[Theorem 1]{MR951969}. We observe that, unlike Baker, we are not able to use normal family arguments in the quasi-Fatou components.

\subsection{The fast escaping set}
A key tool in the proof of the above results is the \emph{fast escaping set}. This was first defined, for a {\tef}, in \cite{MR1684251}, and a detailed study of this set was given in \cite{Rippon01102012}. See also \cite{MR3215194, MR3265357}, which studied the fast escaping set of a quasiregular map of $\R^d$ of transcendental type.  When~$f$ is a function defined on $\C$ or $\R^d$, the fast escaping set is roughly the set of points~$x$ for which $|f^n(x)|$ eventually grows faster than some iterated maximum modulus. Here this refers to iteration of the \emph{maximum modulus} $M(r,f):=\max_{|x|=r}|f(x)|$  as a function of $r>0$.

One motivation for studying the fast escaping set is its intimate connection to the Julia set. In particular, for transcendental entire functions, the boundary of the fast escaping set is the Julia set. The same result also holds for suitable analytic self-maps of $\C^*$ and for many quasiregular self-maps of $\R^d$. We aim to establish the analogous result for quasiregular maps of $S$-transcendental type; see Theorem~\ref{theo:AJ} below.

The fast escaping set of a transcendental analytic self-map of $\C^*$ was first studied in \cite{DMP1}; indeed, our work regarding the fast escaping set is, in a sense, a generalisation of the results of \cite{DMP1} to quasiregular maps of $S$-transcendental type. Our method of definition cannot match that used in \cite{DMP1} exactly, since that definition is given using the maximum and minimum modulus functions, and these are less useful in our setting. Instead we define a family of functions each of which is, in some sense, a generalised maximum modulus function.

Recall that our set of punctures is $S=\{y_0, y_1, \ldots, y_\nu\}$, with $y_0=\infty$. Let $\mathcal{P} := \{0, 1, \ldots, \nu\}$.  For each $j \in \mathcal{P}$, we  define the \emph{generalised modulus function} on $\widehat{\R^d}$ by
$$|x|_j : = \begin{cases} |x|, &\mbox{if }j=0, \\ \dfrac{1}{|x-y_j|}, &\mbox{if }j>0. \end{cases}  $$

Next, we fix $\rho_S > 0$ sufficiently large that
\begin{equation}
\label{rhoSdef}
\{ x \in \widehat{\R^d} : |x|_j \geq \rho_S \} \cap S = \{y_j \}, \qfor j \in\mathcal{P}.
\end{equation}
Note that this definition of $\rho_S$ will be in place throughout the paper. Then, for each $j, k \in \mathcal{P}$, we define the \emph{generalised maximum modulus function} by
\begin{equation}
\label{mukdef}
M_{j,k}(r, f) := \max_{\{ x : |x|_j = r\}} |f(x)|_k, \qfor r>\rho_S.
\end{equation}
Roughly speaking, the function $M_{j,k}$ considers the maximum size of $f$ (compared to the $k$th essential singularity) considered at points near the $j$th essential singularity.

We shall now briefly outline the idea behind the definition of the fast escaping set in our present setting; we defer the precise definition to Section~\ref{section:A}.
It is useful to let $\No := \N \cup \{0\}$ and to call sequences $e=e_0 e_1 \ldots \in \mathcal{P}^{\No}$ \emph{itineraries}. Suppose that $R > 0$ and $e\in \mathcal{P}^{\No}$. We define the \emph{maximum modulus sequence for $e$ starting at $R$} by first setting $R_0 = R$, and then letting
\begin{equation}
\label{Rndef}
R_n := M_{{e_{n-1}},e_{n}}(R_{n-1},f), \qfor n\in\N.
\end{equation}
(We shall see in Section~\ref{section:A} that if $R$ is sufficiently large, then we can guarantee that $R_{n-1}>\rho_S$, so that $R_n$ can indeed be defined by \eqref{Rndef}.) For a given itinerary $e$, the \emph{little fast escaping set} $A_e(f)$ is roughly the set of points $x$ for which $|f^n(x)|_{e_n}$ grows faster than some maximum modulus sequence for $e$. The \emph{fast escaping set} is then defined to be the union
\[ A(f):= \bigcup_{e \in \mathcal{P}^{\No}} A_e(f). \]

\begin{theorem}
\label{theo:A1}
Suppose that $f : \widehat{\R^d}\setminus S \to \widehat{\R^d} \setminus S$ is a quasiregular map of \mbox{$S$-transcendental} type and that $e \in \mathcal{P}^{\No}$. Then $A_e(f)$ is non-empty and all components of $A_e(f)$ are $S$-unbounded.
\end{theorem}

Our second result concerning the little fast escaping sets provides the crucial connection between these sets and the Julia set.
\begin{theorem}
\label{theo:AJ}
Suppose that $f : \widehat{\R^d}\setminus S \to \widehat{\R^d} \setminus S$ is a quasiregular map of \mbox{$S$-transcendental} type and that $e \in \mathcal{P}^{\No}$. Then
\begin{equation}
\label{JAeq}
J(f) = \partial A_e(f) = \partial A(f).
\end{equation}
\end{theorem}
Although there is no assumption of normality in a component of $QF(f)$, the following easy corollary of Theorem~\ref{theo:AJ} is a type of normality property, and is central to our arguments.
\begin{corollary}
\label{corr:QF}
Suppose that $f : \widehat{\R^d}\setminus S \to \widehat{\R^d} \setminus S$ is a quasiregular map of \mbox{$S$-transcendental} type. If  $e \in \mathcal{P}^{\No}$ and $U$ is a quasi-Fatou component of $f$ that meets $A_e(f)$, then $U \subset A_e(f)$.
\end{corollary}

We will see later (Theorem~\ref{theo:A2}\ref{aequiv}) that there are uncountably many disjoint little fast escaping sets. Corollary~\ref{corr:QF} then yields the following result, as the open set $QF(f)$ has only countably many components.

\begin{corollary}
\label{corr:J}
Suppose that $f : \widehat{\R^d}\setminus S \to \widehat{\R^d} \setminus S$ is a quasiregular map of \mbox{$S$-transcendental} type. Then $A(f)\cap J(f) \ne \emptyset$. Moreover, there exists $e \in \mathcal{P}^{\No}$ such that $A_e(f) \subset J(f)$.
\end{corollary}

\subsection{Structure of this paper}
The structure of this paper is as follows. First, in Section~\ref{section:defs}, we give a number of important definitions and background results. Next, in Section~\ref{section:julia}, we prove the first four parts of Theorem~\ref{theo:J}, and also Theorem~\ref{theo:Jclassical}. Since the properties of the fast escaping set are required in the rest of the paper, in Section~\ref{section:A} we give the precise definition of this set and prove Theorem~\ref{theo:A1}. In Section~\ref{section:AJ} we prove Theorem~\ref{theo:AJ} and then, in Section~\ref{section:julia2}, we use it to prove the last part of Theorem~\ref{theo:J}. Finally, we prove Theorem~\ref{theo:F} in Section~\ref{section:fatou}.
%
%
%
%
%
\section{Definitions and background results}
\label{section:defs}
\subsection{Quasiregular maps}
We refer to \cite{MR1238941, MR950174} for a detailed treatment of quasiregular maps. Here we merely recall some definitions and properties used in this paper.

Suppose that $d\geq 2$, that $G \subset \R^d$ is a domain, and that $1 \leq p < \infty$. The \emph{Sobolev space} $W^1_{p,loc}(G)$ consists of those functions $f : G \to \R^d$ for which all first order weak partial derivatives exist and are locally in $L^p$. We say that $f$ is \emph{quasiregular} if $f \in W^1_{d,loc}(G)$ is continuous, and there exists $K_O \geq 1$ such that
\begin{equation}
\label{KOeq}
\left(\sup_{|h|=1} |Df(x)(h)|\right)^d \leq K_O J_f(x) \quad \text{a.e}.
\end{equation}
Here $D f(x)$ denotes the derivative, and $J_f (x)$ denotes the Jacobian determinant. If $f$ is quasiregular, then there also exists $K_I \geq 1$ such that
\begin{equation}
\label{KIeq}
K_I \left( \inf_{|h|=1} |Df(x)(h)| \right)^d \geq J_f(x) \quad \text{a.e}.
\end{equation}
The smallest constants $K_O$ and $K_I$ for which (\ref{KOeq}) and (\ref{KIeq}) hold are denoted by $K_O (f)$ and $K_I (f)$ respectively. If $\max\{K_I (f),K_O (f)\} \leq K$, for some $K \geq 1$, then we say that $f$ is \emph{$K$-quasiregular}.

If $f$ and $g$ are quasiregular maps, and $f$ is defined in the range of $g$, then $f \circ g$ is quasiregular and \cite[Theorem II.6.8]{MR1238941}
\begin{equation}
\label{Keq}
K_I(f \circ g) \leq K_I(f)K_I(g).
\end{equation}

Many properties of analytic functions extend to quasiregular maps; we use, without comment, the fact that a non-constant quasiregular map is discrete and open. We also use the following \cite[Theorem~1.2]{MR583633}, which is Rickman's analogue of Picard's great theorem.

\begin{lemma}
\label{greatpicard}
If $d \geq 2$ and $K \geq 1$, then there exists an integer $q = q(d,K)$ with the following property. If $a_1, \ldots , a_q \in \R^d$ are distinct and $r > 0$, then no K-quasiregular map $f : \{ x \in \R^d : |x| > r \} \to \R^d \setminus \{a_1, \ldots, a_q\}$ has an essential singularity at infinity.
\end{lemma}
The number $q(d,K)$ is called \emph{Rickman's constant}.
\subsection{The capacity of a condenser}
An important tool in the study of quasiregular maps is the capacity of a condenser, and we recall this idea very briefly. If $A \subset \R^d$ is open, and $C \subset A$ is non-empty and compact, then the pair $(A,C)$ is called a \emph{condenser}. Its \emph{capacity}, denoted by cap$(A,C)$, is defined by
$$
\operatorname{cap}(A,C) := \inf_u \int_A |\nabla u|^d dm.
$$
Here, the infimum is taken over all non-negative functions $u \in C^\infty_0(A)$ that satisfy $u(x) \geq 1$, for $x \in C$.

If cap$(A,C) = 0$ for a bounded open set $A$ containing $C$, then cap$(A',C) = 0$ for every bounded open set $A'$ containing $C$; see \cite[Lemma III.2.2]{MR1238941}. In this case we say that $C$ \emph{has zero capacity}, and write cap $C = 0$; otherwise we say that $C$ \emph{has positive capacity}, and write cap~$C > 0$. For an unbounded closed set $C\subset\R^d$, we say that $C$ has zero capacity if every compact subset of $C$ has zero capacity. Roughly speaking, cap $C = 0$ means that $C$ is a ``small'' set. In particular, it is well-known that any finite set has zero capacity.
\subsection{The modulus of a curve family}\label{sect:modulus}
The proof of Theorem~\ref{theo:AJ} closely follows the proof of \cite[Theorem 1.2]{MR3265357}. This requires us to introduce the concept of the modulus of a curve family, although we are able to eschew all detail and refer to \cite{MR1238941,MR950174} for more information. If $\Gamma$ is a family of paths in $\R^d$, then a non-negative Borel function $\rho : \R^d \to \R \cup \{\infty\}$ is called \emph{admissible} if $\int_\gamma \rho \ ds \geq 1$, for all locally rectifiable paths $\gamma \in \Gamma$. We let $\mathcal{F}(\Gamma)$ be the family of all admissible Borel functions, and let the \emph{modulus} of $\Gamma$ be defined by
$$
M(\Gamma) := \inf_{\rho \in \mathcal{F}(\Gamma)} \int_{\R^d}\rho^d \ dm.
$$

Finally, if $G \subset \R^d$ is a domain, and $E, F$ are subsets of $\overline{G}$, then we denote by $\Delta(E,F;G)$ the family of all paths which have one endpoint in $E$, one endpoint in $F$, and which otherwise are in $G$.
\subsection{Topological prerequisites}
For simplicity we define a \emph{continuum} as a non-empty subset of $\widehat{\R^d}$ which is compact and connected in $\widehat{\R^d}$. We use the following version of \cite[Theorem 2 p.172]{MR0259835}, in which we take closures in $\widehat{\R^d}$; this is known as a boundary bumping theorem.
\begin{proposition}
\label{kurat}
Suppose that $X$ is a proper subset of a continuum $K$, and that $X'$ is a connected component of $X$. Then $\overline{X'} \cap \overline{K \setminus X} \ne \emptyset.$
\end{proposition}
We also need the following version of \cite[Lemma 3.2]{MR2448586}. That result, roughly speaking, is the case of a single puncture at infinity.
\begin{proposition}
\label{bflmpunctured}
Suppose that $f : \widehat{\R^d}\setminus S \to \widehat{\R^d} \setminus S$ is a quasiregular map of \mbox{$S$-transcendental} type. Suppose that $E \subset \widehat{\R^d}$ is a continuum that meets $S$. Then all components of $f^{-1}(E)$ are $S$-unbounded.
\end{proposition}
\begin{proof}
Let $F=f^{-1}(E)\cup S$, and note that this is a compact subset of $\widehat{\R^d}$. If every component of $F$ meets $S$, then the result follows by applying Proposition~\ref{kurat} with $X'$ any component of $f^{-1}(E)$, and $X=K\setminus S$, where $K$ is the component of $F$ containing $X'$.

Otherwise, some component of $F$ does not meet $S$ and hence $F$ can be partitioned into two non-empty disjoint relatively closed sets $H_1$ and $H_2$ such that $S\subset H_2$. A contradiction can then be deduced in a very similar way to \cite[Lemma 3.2]{MR2448586}; the details are omitted.
\end{proof}

We also require the following, which is a version of a result established, for example, in the proof of \cite[Theorem 1.4]{SixsmithNicks1}.

\begin{proposition}
\label{prop:comps}
Suppose that $X \subset \widehat{\R^d}\setminus S$ is closed, and that $K$ is a component of $X$ which is $S$-bounded. Then there is an $S$-bounded domain $V$ such that $K \subset V$ and $\partial V \cap X = \emptyset$.
\end{proposition}
\begin{proof}
Note that $X\cup S$ is a compact subset of $\widehat{\R^d}$. Let $T$ be the component of $X\cup S$ that contains $K$. Suppose that $T$ meets $S$, in which case $T\setminus S$ is a proper subset of $T$. Since $K$ is a connected component of $T\setminus S$, Proposition~\ref{kurat} implies that $K$ is $S$-unbounded. This is a contradiction.

Hence $T$ is disjoint from $S$. Since $T$ is a component of the compact set $X \cup S$, it follows from \cite[Theorem~5.6]{MR0132534} that $X \cup S$ can be partitioned into two disjoint relatively closed sets $H_1$ and $H_2$ such that $T\subset H_1$ and $S\subset H_2$. The sets $H_1$ and $H_2$ are closed (and so compact) in~$\widehat{\R^d}$, and hence, using the spherical metric  on~$\widehat{\R^d}$, there is a positive distance $2\varepsilon$ between them. Let $V'$ be the $\varepsilon$-neighbourhood of $H_1$ and note that this is $S$-bounded. We take $V$ to be the component of $V'$ that contains $T$.
\end{proof}

The final result in this subsection concerns continuous functions, and is almost identical to \cite[Lemma 2.6]{danslow}; see also \cite[Lemma 1]{MR2792984}. The proof is omitted.
\begin{lemma}
\label{lemm:exists}
Suppose that $f : \widehat{\R^d}\setminus S \to \widehat{\R^d}\setminus S$ is a continuous function, and that $(E_n)_{n\in\No}$ is a sequence
of non-empty $S$-bounded subsets of $\widehat{\R^d}\setminus S$ such that
$$
f(E_n) \supset E_{n+1}, \qfor n\in\No.
$$
Then there exists $\xi\in \widehat{\R^d}\setminus S$ such that $f^n(\xi) \in \overline{E_n}$, for $n\in\No$.
\end{lemma}
%
%
%
\subsection{Properties of the generalised maximum modulus functions}
It is useful to define some M\"{o}bius maps that will be referred to in several subsequent proofs. First let $\tau$ be the reflection defined by $\tau(\infty) = \infty$ and
$$
\tau(x) = \tau(x_1, x_2, \ldots, x_d) = (-x_1, x_2, \ldots, x_d), \qfor x \in \R^d.
$$
We then let $\phi_0 : \widehat{\R^d} \to \widehat{\R^d}$ be the identity map, and, for $j \in \mathcal{P}\setminus\{0\}$, we let $\phi_j: \widehat{\R^d} \to \widehat{\R^d}$ be the M\"{o}bius map defined by
$$
\phi_j(x) := \tau\left(\frac{x - y_j}{|x - y_j|^2}\right).
$$
Note that the function $\tau$ is introduced only to ensure that each $\phi_j$ is orientation-preserving.
 The usefulness of the maps $\phi_j$ is due to the fact that
\[ |x|_j = |\phi_j(x)|, \qfor x\in\widehat{\R^d}, \ j \in \mathcal{P}. \] 


In the following lemma we gather various properties of the generalised maximum modulus functions. We later use the first part of this result without further comment.
\begin{lemma}
\label{lemma:gen}
Suppose that $f : \widehat{\R^d}\setminus S \to \widehat{\R^d} \setminus S$ is a quasiregular map of \mbox{$S$-transcendental} type. Suppose also that $j, k \in \mathcal{P}$. Then:
\begin{enumerate}[label=(\alph*)]
\item\label{Mnormal} $M_{j,k}(r, f)$ is increasing, for sufficiently large values of $r > \rho_S$.
\item\label{Mmult} If $A > 1$, then
$$
\lim_{r\rightarrow\infty} \frac{M_{j,k}(Ar,f)}{M_{j,k}(r,f)} = \infty.
$$
\item\label{Mbig} We have
$$
\lim_{r\rightarrow\infty} \frac{\log M_{j,k}(r,f)}{\log r} = \infty.
$$
\end{enumerate}
\end{lemma}
\begin{proof}
Suppose that $f : \widehat{\R^d}\setminus S \to \widehat{\R^d} \setminus S$ is a quasiregular map of \mbox{$S$-transcendental} type. Part \ref{Mnormal} is a consequence of the fact that $f$ is an open map. For part \ref{Mmult}, we consider the quasiregular map $g = \phi_k \circ f \circ \phi_j^{-1}$. Note that $g$ is defined on a punctured neighbourhood of infinity, and has an essential singularity at infinity. We need to show that
\begin{equation}
\label{anMequation}
\lim_{r\rightarrow\infty} \frac{M(Ar,g)}{M(r,g)} = \infty.
\end{equation}

The proof of this fact is similar to the proof of \cite[Lemma 3.3]{MR2248829}, and is omitted. Part \ref{Mbig} is a simple consequence of part \ref{Mmult}.
\end{proof}
\subsection{Asymptotic values}
Suppose that $f : \widehat{\R^d}\setminus S \to \widehat{\R^d} \setminus S$ is a quasiregular map of \mbox{$S$-transcendental} type. If $a \in \widehat{\R^d}$ and $j \in \mathcal{P}$, then we say that \emph{$a$ is an asymptotic value of $f$ at the $j$th puncture} if there is a curve $\gamma : (0,1) \to \widehat{\R^d}\setminus S$ such that $\gamma(t)\rightarrow y_j$ and $f(\gamma(t)) \rightarrow a$ as $t\rightarrow 1$.

The following is an immediate consequence of \cite[Theorem~VII.2.6]{MR1238941}.
\begin{lemma}
\label{lemma:asympt}
Suppose that $f : \widehat{\R^d}\setminus S \to \widehat{\R^d} \setminus S$ is a quasiregular map of \mbox{$S$-transcendental} type, and  that $j, k \in \mathcal{P}$. Then $y_k$ is an asymptotic value of $f$ at the $j$th puncture.
\end{lemma}
%
%
%
%
%
%
\section{Proof of Theorem~\ref{theo:Jclassical} and the first four parts of Theorem~\ref{theo:J}}
\label{section:julia}
We begin by making the following simple observation.
\begin{proposition}
\label{prop:useful}
Suppose that $f : \widehat{\R^d}\setminus S \to \widehat{\R^d} \setminus S$ is a quasiregular map of \mbox{$S$-transcendental} type, and that $U$ is a neighbourhood of a point $x \in J(f)$. Then
\begin{equation*}
\widehat{\R^d} \setminus O^+(U) \subset E(f).
\end{equation*}
\end{proposition}
\begin{proof}
Suppose that $y \in \widehat{\R^d}\setminus O^+(U)$. It is easy to see that $f^{-1}(y) \subset \widehat{\R^d}\setminus O^+(U)$. It follows that $O^-(y)$ is contained in the finite set $\widehat{\R^d}\setminus O^+(U)$, and thus $y \in E(f)$ as required.
\end{proof}
It is useful to define a set which we later show is, in fact, equal to the Julia set.
\begin{definition}\normalfont
\label{jcapdef}
For a quasiregular map $f : \widehat{\R^d}\setminus S \to \widehat{\R^d} \setminus S$ of \mbox{$S$-transcendental} type, we denote by $J_{\operatorname{cap}}(f)$ the set of all points $x \in \widehat{\R^d}\setminus S$ such that, for every neighbourhood $U$ of $x$, we have
\begin{equation}
\label{capcond}
\operatorname{cap }\widehat{\R^d} \setminus O^+(U) = 0.
\end{equation}
\end{definition}
\begin{remark}\normalfont
For \emph{general} quasiregular self-maps of $\R^d$ or $\widehat{\R^d}$, a capacity condition such as that in (\ref{capcond}) is used to \emph{define} the Julia set; see \cite{MR3009101, MR3265283}. In those settings it remains a significant open problem to determine whether this is equivalent to a ``finite omitted set definition'' of the Julia set such as that in Definition~\ref{jdef}.
\end{remark}
Since, by Lemma~\ref{lemma:asympt}, self-maps of punctured space have at least one finite asymptotic value at each puncture, we can make progress through the following  four propositions.
\begin{proposition}
\label{danprop1}
Suppose that $d \geq 2$ and $K \geq 1$, and let $q \in \N$ be Rickman's constant. Then there exists an integer $N > q$ with the following property. If $U_1, \ldots, U_N$ are $S$-bounded domains in $\widehat{\R^d}\setminus S$ with pairwise disjoint closures, and if $f : \widehat{\R^d}\setminus S \to \widehat{\R^d}\setminus S$ is a K-quasiregular map of $S$-transcendental type such that, for each $j \in \{1, \ldots, N\}$,
$$
f(U_j) \supset U_i, \qfor\text{at least } N-q \text{ values of } i \in \{1, \ldots, N\},
$$
then the following both hold.
\begin{enumerate}[label=(\alph*)]
\item We have that $\overline{U_j} \cap J_{\operatorname{cap}}(f) \ne \emptyset$, for $j \in \{1, \ldots, N\}$.\label{apart}
\item There exists $j^* \in \{1, \ldots, N\}$ such that $\operatorname{cap }\overline{O^-(y)} > 0$, for $y \in \overline{U_{j^*}}$.\label{bpart}
\end{enumerate}
\end{proposition}
\begin{proof}
The proof of these results is almost identical to certain proofs in \cite{MR3265283}, with the function $g_m$ referred to in \cite{MR3265283} replaced by $f$. For reasons of brevity, we have not reproduced all the details here.

Choose $N\in\N$, divisible by $4$, such that $N > \max\{4K, 8q\}$. Then \ref{apart} can be deduced from \cite[Theorem 3.2]{MR3009101} and the definition of $J_{\operatorname{cap}}(f)$, by the argument used on \cite[p.161]{MR3265283} in the ``Proof of Theorem 1.1 for functions without the pits effect''.

The proof of \cite[Theorem 1.9]{MR3265283} then shows that $3N/4$ of the domains $U_j$ have the property that cap $\overline{O^-(y)} > 0$, for $y \in \overline{U_j}$.
\end{proof}
The next proposition is analogous to \cite[Lemma 5.1]{danslow}. Here if $U \subset \R^d$ and $r > 0$, then we define $rU := \{ rx : x \in U\}$. It is useful to define a topological ring, for $j \in \mathcal{P}$ and $0 < r_1 < r_2$, by
$$
A_j(r_1, r_2) := \{ x \in \R^d : r_1 < |x|_j < r_2 \},
$$
and also, for simplicity, a ``standard'' ring $$A(r_1, r_2) := A_0(r_1, r_2) = \{ x \in \R^d : r_1 < |x| < r_2 \}.$$
\begin{proposition}
\label{danprop2}
Suppose that $f : \widehat{\R^d}\setminus S \to \widehat{\R^d} \setminus S$ is a $K$-quasiregular map of \mbox{$S$-transcendental} type. Suppose that $\alpha > 1$, that $N$ is an integer greater than $q$, where $q$ is Rickman's constant, and that $U_1', \ldots, U_N'$ are non-empty subsets of $\R^d$ with pairwise disjoint closures in $\widehat{\R^d}$. Then there exists $r_0 > \rho_S$ such that, for all $r \geq r_0$ and $1 \leq R \leq M(r, f)$,
$$
f(A(r, \alpha r)) \supset R U_i', \qfor \text{at least } N-q \text{ values of } i \in \{1, \ldots, N\}.
$$
In particular, if $x_1, \ldots x_{q+1} \in \R^d$ are distinct, then, for all sufficiently large $r > \rho_S$, there exists $y \in A(r, \alpha r)$ such that $f(y) \in \{x_1, \ldots, x_{q+1}\}$.
\end{proposition}
\begin{proof}
Suppose, by way of contradiction, that no such $r_0$ exists. Then there exist sequences of real numbers $(r_k)_{k\in\N}$ and $(R_k)_{k\in\N}$ such that $(r_k)_{k\in\N}$ increases to $\infty$ and $1 \leq R_k \leq M(r_k, f)$, but $f(A(r_k, \alpha r_k))$ does not contain $R_k U_i'$ for at least $q+1$ choices of $i\in\{1, \ldots, N\}$. We can assume that $r_1 > \rho_S$, and so $$\{ x \in \R^d : |x| > r_1\} \cap S = \emptyset.$$

For each $k \in\N$, let $f_k : A(1, \alpha) \to \R^d$ be the $K$-quasiregular map defined by $$f_k(x) := \frac{f(r_k x)}{R_k}.$$ Then $f_k$ omits a point in $U_i'$ for at least $q+1$ values of $i$. We deduce that $\{ f_k \}_{k \in\N}$ is a normal family on $A(1, \alpha)$ by the quasiregular analogue of Montel's theorem due to Miniowitz \cite[Theorem 5]{MR633273}.

On the other hand, the existence of a finite asymptotic value for $f$, by Lemma~\ref{lemma:asympt}, means that there exists $c>0$ such that $$\min_{|x|=r}|f(x)| \leq c, \qfor r\geq r_1.$$ Hence,
\begin{equation}
\label{danseq2}
\min\left\{|f_k(x)| : |x| = \frac{1+\alpha}{2}\right\} \leq \frac{c}{R_k} \leq c, \qfor k \in \N,
\end{equation}
while also
\begin{equation}
\label{danseq3}
M\left(\frac{1+\alpha}{2}, f_k\right) = \frac{M\left(\frac{1+\alpha}{2}r_k,f\right)}{R_k} \geq \frac{M\left(\frac{1+\alpha}{2}r_k,f\right)}{M(r_k, f)}, \qfor k \in \N.
\end{equation}

By Lemma~\ref{lemma:gen}\ref{Mmult} this last term tends to infinity as $k \rightarrow\infty$. Therefore (\ref{danseq2}) and (\ref{danseq3}) together contradict the normality of the family $\{f_k\}_{k \in\N}$.

The final observation of the proposition follows by taking $N=q+1$, $R=1$ and $U_i' = \{ x_i \}$, for $i \in \{1, \ldots, N\}$.
\end{proof}

\begin{proposition}
\label{prop:capclaim}
Suppose that $f : \widehat{\R^d}\setminus S \to \widehat{\R^d} \setminus S$ is a quasiregular map of \mbox{$S$-transcendental} type. Then $J_{\operatorname{cap}}(f)$ is infinite and
\begin{equation}
\label{dancapclaim}
\operatorname{cap }\overline{O^-(x)} > 0, \qfor x \in \R^d \setminus E(f).
\end{equation}
\end{proposition}

\begin{proof}
Suppose that $f : \widehat{\R^d} \setminus S \to \widehat{\R^d} \setminus S$ is a $K$-quasiregular map of $S$-transcendental type. Take $N \in \N$ as in Proposition~\ref{danprop1}, let $\alpha=2$, and let $$U_i' := A(3^i, 3^i\alpha), \qfor i \in \{1, \ldots, N\}.$$

By Lemma~\ref{lemma:gen}\ref{Mbig}, let $s_0 > \rho_S$ be sufficiently large that $M(3s, f) \geq s$, for $s \geq s_0$. Let $r_0$ be the constant from Proposition~\ref{danprop2}. For $s \geq \max\{r_0,s_0\}$, we put $r = 3^js$, $R=s$, and $U_i := s U_i'$, for $i \in \{1, \ldots, N\}$. Proposition~\ref{danprop2} then yields that, for each $j \in \{1, \ldots, N\}$,
$$
f(U_j) = f(A(3^js, 3^j\alpha s)) \supset U_i, \qfor \text{at least } N-q \text{ values of } i\in \{1, \ldots, N\}.
$$

Next, Proposition~\ref{danprop1}\ref{apart} implies that $\overline{U_{j}} \cap J_{\operatorname{cap}}(f) \ne \emptyset$, for $j \in \{1, \ldots, N\}$. By choosing arbitrarily large values of $s$, it then follows that $J_{\operatorname{cap}}(f)$ is infinite.

To prove \eqref{dancapclaim}, take $x \in \R^d\setminus E(f)$ and recall that, by definition, $O^-(x)$ is infinite. It follows by the last part of Proposition~\ref{danprop2} that there exists $r' = r'(x) > \rho_S$ such that for all $r \geq r'$,
\begin{equation}
\label{daneq5}
\text{there exists } y \in A(r, 2r) \text{ such that } f(y) \in O^-(x).
\end{equation}

Now choose a large value of $s \geq r'$, and take $j^* \in \{1, \ldots, N\}$ as given by Proposition~\ref{danprop1}\ref{bpart}. It then follows from (\ref{daneq5}) that there exists $y \in U_{j^*}$ such that $f(y) \in O^-(x)$. This implies that $y \in O^-(x)$, and hence that $O^-(y) \subset O^-(x)$. Therefore, Proposition~\ref{danprop1}\ref{bpart} yields cap $\overline{O^-(x)} > 0$.
\end{proof}

\begin{proposition}\label{J=J_cap}
Suppose that $f : \widehat{\R^d}\setminus S \to \widehat{\R^d} \setminus S$ is a quasiregular map of \mbox{$S$-transcendental} type. 
Then
\[  J(f) = J_{\operatorname{cap}}(f) \subset \overline{O^-(x)}, \qfor x \in \R^d \setminus E(f). \]
\end{proposition}

\begin{proof}
This is very similar to \cite[Proof of Theorem 5.1]{MR3265283}. 
It follows from the definitions that $J(f) \subset J_{\operatorname{cap}}(f)$. Let $x \in \widehat{\R^d}\setminus E(f)$, and let $U$ be an open set intersecting $J_{\operatorname{cap}}(f)$. By Proposition~\ref{prop:capclaim}, cap $\overline{O^-(x)} > 0$ and so $O^+(U) \cap \overline{O^-(x)} \ne \emptyset$, by the definition of $J_{\operatorname{cap}}(f)$. Since $O^+(U)$ is open and $\overline{O^-(x)}$ is closed, we have in fact that $O^+(U) \cap O^-(x) \ne \emptyset$. This in turn implies both that
\begin{equation}
\label{daneq6}
x \in O^+(U),
\end{equation}
and
\begin{equation}
\label{daneq7}
U \cap O^-(x) \ne \emptyset.
\end{equation}

It follows from (\ref{daneq6}), and from the fact that $x \in \widehat{\R^d} \setminus E(f)$ was arbitrary, that $\widehat{\R^d} \setminus O^+(U) \subset E(f)$. Since this holds for any neighbourhood $U$ of any point of $J_{\operatorname{cap}}(f)$, and since $E(f)$ is finite, this shows that $J_{\operatorname{cap}}(f) \subset J(f)$. It follows that $J(f) = J_{\operatorname{cap}}(f)$ and, in particular, $J(f)$ is infinite.
Since (\ref{daneq7}) holds for every open set $U$ intersecting $J(f)$, we can also deduce that $J(f) \subset \overline{O^-(x)}$.
\end{proof}

 We are now ready to prove the first four parts of Theorem~\ref{theo:J}. 

\begin{proof}[Proof of Theorem~\ref{theo:J}\ref{jperfect}--\ref{jconn}]
We note first that $J(f)$ is infinite by Propositions~\ref{prop:capclaim} and~\ref{J=J_cap}. Moreover, part \ref{jback} follows from Proposition~\ref{J=J_cap} since $J(f)$ is closed and completely invariant.
We then complete the proof of part~\ref{jperfect} by using a standard argument to deduce from part~\ref{jback} that $J(f)$ is perfect. See, for example, \cite[Proof of Theorem 5.1(iv)]{MR3265283}. \\

Next we prove part \ref{jnice}, which states that $J(f) = J(f^p)$, for $p\in\N$. It is not hard to see that $E(f) = E(f^p)$. It is clear that $J(f^p) \subset J(f)$, and so we need to demonstrate the reverse inclusion. Our proof is similar to the proof of \cite[Theorem~5.2]{MR3265283}. In fact we will show that $J(f) \setminus E(f) \subset J(f^p)$; this is sufficient because $E(f)$ is finite, $J(f)$ is perfect and $J(f^p)$ is closed.

Suppose that $x \in J(f) \setminus E(f)$, and let $U$ be a neighbourhood of $x$ disjoint from $E(f)$. Then
\begin{equation}
\label{daneq8}
f^N(f^{-N}(U)) = U, \qfor N \in \N.
\end{equation}

Since $J(f^p)$ is infinite, and $U$ meets $J(f)$, there exists $m \in \N$ with the property that $f^m(U) \cap J(f^p) \ne \emptyset$. Let $V := f^m(U)$ and $m = pk - \ell$, with $k \in \N$ and $\ell \in \{0, 1, \ldots, p-1\}$. Then $V = f^{pk}(f^{-\ell}(U))$, by (\ref{daneq8}). Hence
$$
f^{\ell}\left(\bigcup_{n\in\N} f^{pn}(V)\right) \subset f^{\ell}\left(\bigcup_{n\in\N} f^{pn}(f^{-\ell}(U))\right) = \bigcup_{n\in\N} f^{pn}(U).
$$

We deduce that
\begin{align}
\label{daneq9}
\R^d\setminus \bigcup_{n\in\N} f^{pn}(U) &\subset \R^d\setminus f^\ell\left(\bigcup_{n\in\N} f^{pn}(V)\right) \nonumber \\
                                         &\subset \left(f^\ell\left(\R^d\right)\setminus f^\ell\left(\bigcup_{n\in\N} f^{pn}(V)\right)\right) \cup E(f^\ell) \nonumber \\
                                         &\subset f^{\ell}\left(\R^d\setminus\bigcup_{n\in\N} f^{pn}(V)\right) \cup E(f^\ell).
\end{align}

This last set is finite, because $E(f^\ell) = E(f)$ is finite and because, by Proposition~\ref{prop:useful},
$$
\R^d\setminus\bigcup_{n\in\N} f^{pn}(V) \subset E(f^p) = E(f),
$$
since $V$ is open and meets $J(f^p)$.

We have shown that $\R^d\setminus \bigcup_{n\in\N} f^{pn}(U)$ is a finite set, for any sufficiently small neighbourhood, $U$, of $x$. It follows that $x \in J(f^p)$. Thus $J(f) \setminus E(f) \subset J(f^p)$, as required. \\

Finally we prove part \ref{jconn}, which states that either $J(f)$ is connected or it has infinitely many components. Suppose that $J(f)$ has finitely many components $J_1, \ldots, J_n$ with $n\in\N$. Since $J(f)$ is perfect, none of these components is a singleton and hence each is an infinite set. By complete invariance, each image $f(J_j)$ is contained in some component $J_k$. Moreover, each $J_k$ must contain some image $f(J_j)$ because every point in $J(f)\setminus E(f)$ must have a preimage under $f$. Thus $f$ permutes the set of components $\{J_1, \ldots, J_n\}$, and so there exists $p\in\N$ such that $f^p(J_j) \subset J_j$, for $j\in\{1,\ldots,n\}$. Since $J(f^p) = J(f)$, by part \ref{jnice}, we can assume that $p=1$.

Let $G$ be the complement of $J_1$. Then $f(G) \subset G$, and it follows from the definition of $J(f)$ that $G \subset QF(f)$. Therefore, we conclude that $J_1$ is the only component of $J(f)$, as required.
\end{proof}
Finally in this section, we prove that our definition of the Julia set agrees with the classical definition for a transcendental analytic self-map of $\C^*$.
\begin{proof}[Proof of Theorem~\ref{theo:Jclassical}]
Suppose that $f$ is a transcendental analytic self-map of $\C^*$, and let $J_{\operatorname{class}}(f)$ denote the (classical) Julia set. Note that we have $E(f) = \{0, \infty\}$. Choose any point $z \in J_{\operatorname{class}}(f) \setminus \{0, \infty\}$. The fact that $z \in J(f)$ follows from  Montel's theorem. 
It then follows that $J(f) = \overline{O^-(z)} = J_{\operatorname{class}}(f)$, by Theorem~\ref{theo:J}\ref{jback} and the $\C^*$ analogue of this result (see \cite{Bhattacharyya}).
\end{proof}
%
%
%
%
%
\section{The fast escaping set}
\label{section:A}
 
 In this section we give the precise definitions of the little fast escaping sets, and the fast escaping set, of a quasiregular map $f$ of $S$-transcendental type. We then establish some fundamental properties of these sets before proving Theorem~\ref{theo:A1}.

We begin by considering the maximum modulus sequences that were mentioned in the introduction. Given an itinerary $e \in \mathcal{P}^{\No}$ and a sufficiently large $R>0$, the following lemma allows us to define the maximum modulus sequence for $e$ starting at $R$ by setting $R_0 = R$ and
\begin{equation}
\label{Rndef2}
R_n := M_{{e_{n-1}},e_{n}}(R_{n-1},f), \qfor n\in\N.
\end{equation}

\begin{lemma}
\label{lemm:mmseq}
Suppose that $f : \widehat{\R^d}\setminus S \to \widehat{\R^d} \setminus S$ is a quasiregular map of \mbox{$S$-transcendental} type. Then there exists $R(f) > \rho_S$ such that if $R > R(f)$ and $e \in \mathcal{P}^{\No}$, then the maximum modulus sequence for $e$ starting at $R$, denoted $(R_n)_{n\in\No}$, can be defined as above and the following hold:
\begin{enumerate}[label=(\alph*)]
\item $M_{j,k}(r, f) > r^2$, for $j,k \in \mathcal{P}$ and $r > R(f)$;\label{Misbig}
\item the sequence $(R_n)_{n\in\No}$ is strictly increasing and $R_n \rightarrow\infty$ as $n\rightarrow\infty$;\label{seqgrows} and
\item if $R' > R$ and $(R'_n)_{n\in\No}$ is the maximum modulus sequence for $e$ starting at $R'$, then $R'_n > R_n$, for $n\in\No$.\label{seqbigger}
\end{enumerate}
\end{lemma}

\begin{proof}
Suppose that $f : \widehat{\R^d}\setminus S \to \widehat{\R^d} \setminus S$ is a quasiregular map of \mbox{$S$-transcendental} type. By Lemma~\ref{lemma:gen}\ref{Mbig}, and since $\mathcal{P}$ is finite, we can choose $R(f) > \max\{1,\rho_S\}$ sufficiently large that $M_{j,k}(r,f) > r^2$, for $j, k \in \mathcal{P}$ and $r > R(f)$. Parts \ref{Misbig} and \ref{seqgrows} follow.  In particular, the sequence $(R_n)$ can be defined by \eqref{Rndef2} as each term is greater than $\rho_S$. Increasing $R(f)$ if necessary, part \ref{seqbigger} follows from Lemma~\ref{lemma:gen}\ref{Mnormal}.\end{proof}

For brevity, we refer throughout this paper to the constant $R(f)$ in Lemma~\ref{lemm:mmseq} simply as $R(f)$.

Now, suppose that $R > R(f)$ and that $e \in \mathcal{P}^{\No}$, and let $(R_n)_{n\in\No}$ be the maximum modulus sequence for $e$ starting at $R$. In a similar way to \cite{DMP1}, for each $\ell \in \Z$ we define \emph{closed} sets -- the \emph{little level sets} -- by
$$
A_e^{\ell}(f, R) := \{ x : |f^{\ell+n}(x)|_{e_{n}} \geq R_n, \text{ for } n \in\No \text{ such that } n + \ell \geq 0 \}.
$$
A note on this definition. An itinerary $e \in \mathcal{P}^{\No}$ represents a way of approaching $S$ by proximity to a particular sequence $(y_{e_n})_{n\in\No}$ of elements of $S$. Roughly speaking, a point lies in $A_e^\ell(f,R)$ if its iterates tend to $S$ with a certain itinerary, faster than the maximum modulus sequence $(R_n)_{n\in\No}$ grows for the same itinerary.

We denote the \emph{shift map} by $\sigma$; in other words, $\sigma(e_0 e_1 \ldots) = e_1 e_2 \ldots$. We then define the \emph{little fast escaping set $A_e(f)$} by
$$
A_e(f) := \bigcup_{\substack{\ell\in\Z \\ k\in\No}} A_{\sigma^k(e)}^{\ell}(f, R).
$$
Note that we have suppressed the dependence on $R$, because we show in Theorem~\ref{theo:A2}\ref{aindepR} below that this definition is independent of the choice of $R$, provided that $R>R(f)$.
We also stress that the maximum modulus sequence used in the definition of $A_{e}^{\ell}(f, R)$ is not, in general, the same as the maximum modulus sequence used in the definition of, for example, $A_{\sigma(e)}^{\ell}(f, R)$.

Finally we define the \emph{fast escaping set} by
$$
A(f) := \bigcup_{e \in \mathcal{P}^{\No}} A_e(f).
$$

\begin{remarks}\normalfont
Our notation is necessarily rather different to that in \cite{DMP1}. However, it can be shown that for a transcendental analytic self-map of $\C^*$ our definitions of the little fast escaping sets and fast escaping set are equivalent to those of \cite{DMP1}.

We also observe in passing that in the case that $\nu=0$ and $S = \{ \infty\}$, which we are not studying here, the set $\mathcal{P}^{\No}$ contains a single element, and our definition of the fast escaping set coincides with that given in \cite{MR3215194, MR3265357} for a quasiregular map with a single essential singularity at infinity.
\end{remarks}

The following proposition is useful; the proof is almost immediate from the definitions.
\begin{proposition}
\label{ldown}
Suppose that $f : \widehat{\R^d}\setminus S \to \widehat{\R^d} \setminus S$ is a quasiregular map of \mbox{$S$-transcendental} type, that $e \in \mathcal{P}^{\No}$ and that $R > R(f)$. Then
\begin{equation}
\label{prope1}
A_{\sigma^k(e)}^{\ell}(f, R) \subset A_{\sigma^{k+1}(e)}^{\ell+1}(f, R), \qfor \ell\in \Z \text{ and } k\in\No,
\end{equation}
and
\begin{equation}
\label{prope2}
A_e(f) = \bigcup_{\substack{\ell\in\No \\ k\in\No}} A_{\sigma^k(e)}^{\ell}(f, R).
\end{equation}
\end{proposition}
\begin{proof}
Let $(R_n)_{n\in\No}$ be the maximum modulus sequence for $\sigma^k(e)$ starting at $R$, and let $(R'_n)_{n\in\No}$ be the maximum modulus sequence for $\sigma^{k+1}(e)$ starting at $R$. It follows from Lemma~\ref{lemm:mmseq} that $R_{n+1} > R'_n$, for $n\in\No$.

Suppose that $x \in A_{\sigma^k(e)}^{\ell}(f, R)$. It follows that
$$
|f^{\ell+n+1}(x)|_{e_{n+k+1}} \geq R_{n+1} > R'_n, \qfor n\in\No, \ \ell + n + 1 \geq 0.
$$
Hence $x \in A_{\sigma^{k+1}(e)}^{\ell+1}(f, R)$, and (\ref{prope1}) follows. Equation (\ref{prope2}) then follows from the observation that if $\ell < 0$, then (\ref{prope1}) implies that $A_{\sigma^k(e)}^{\ell}(f, R) \subset A_{\sigma^{k-\ell}(e)}^{0}(f, R)$, for $k \in \No$. \end{proof}

Our next result collects a number of facts concerning the fast escaping set and the little fast escaping sets. If $e$ and $e'$ are elements of $\mathcal{P}^{\No}$, then we say that they are \emph{equivalent} if there exist integers $n$ and $m$ such that $\sigma^n(e) = \sigma^m(e')$.

\begin{theorem}
\label{theo:A2}
Suppose that $f : \widehat{\R^d}\setminus S \to \widehat{\R^d} \setminus S$ is a quasiregular map of \mbox{$S$-transcendental} type and that $e, e' \in \mathcal{P}^{\No}$. Then we have the following.
\begin{enumerate}[label=(\alph*)]
\item The definitions of $A_e(f)$ and $A(f)$ are both independent of the choice of $R > R(f)$.\label{aindepR}
\item The sets $A(f)$ and $A_e(f)$ are completely invariant.\label{ainv}
\item The sets $A_e(f)$ and $A_{e'}(f)$ are equal if $e$ and $e'$ are equivalent, but are disjoint otherwise. In particular, $A_e(f) = A_{\sigma(e)}(f)$. \label{aequiv}
\end{enumerate}

\end{theorem}

\begin{proof}
%
To prove part~\ref{aindepR}, we only need to prove that the definition of $A_e(f)$ is independent of the choice of $R> R(f)$, since the fact that the definition of $A(f)$ is independent of the choice of $R> R(f)$ follows from this.
Suppose, without loss of generality, that $R' > R > R(f)$. It follows from Lemma~\ref{lemm:mmseq}\ref{seqbigger} that
\begin{equation}
\label{anincl}
\bigcup_{\substack{\ell\in\Z \\ k\in\No}} A_{\sigma^k(e)}^{\ell}(f, R') \subset \bigcup_{\substack{\ell\in\Z \\ k\in\No}}A_{\sigma^k(e)}^{\ell}(f, R).
\end{equation}

We complete the proof of \ref{aindepR} by showing that the reverse inclusion to (\ref{anincl}) holds. First, by Lemma~\ref{lemm:mmseq}\ref{Misbig}, we choose $\eta\in\N$ sufficiently large that the following holds. If $\tilde{e} \in \mathcal{P}^{\No}$ and $(\rho_n)_{n\in\No}$ is the maximum modulus sequence for $\tilde{e}$ starting at $R$, then $\rho_\eta > R'$.

Let $\tilde{e}$ be the symbol sequence which consists of, for example, $\eta$ zeros followed by the symbols of $e$. Suppose that $\ell \in \Z$ and $k \in\No$. Let $(\rho_n)_{n\in\No}$ be the maximum modulus sequence for $\sigma^{k}(\tilde{e})$ starting at $R$. By the choice of $\eta$, we have that $\rho_\eta > R'$, and so it follows from Lemma~\ref{lemm:mmseq}\ref{seqbigger} and the definition of $\tilde{e}$ that
$$
A_{\sigma^k(e)}^{\ell}(f, R') \supset A_{\sigma^k(e)}^{\ell}(f, \rho_\eta) = A_{\sigma^{k+\eta}(\tilde{e})}^{\ell}(f, \rho_\eta).
$$

We claim that $A_{\sigma^{k+\eta}(\tilde{e})}^{\ell}(f, \rho_\eta) \supset A_{\sigma^{k}(\tilde{e})}^{\ell-\eta}(f, R)$. For, suppose that $x \in A_{\sigma^{k}(\tilde{e})}^{\ell-\eta}(f, R)$. It follows that
$$
|f^{\ell-\eta+n}(x)|_{\tilde{e}_{k+n}} \geq \rho_n, \qfor n\in \No, \ \ell - \eta + n \geq 0,
$$
in which case
$$
|f^{\ell+n}(x)|_{\tilde{e}_{k+\eta+n}} \geq \rho_{n+\eta}, \qfor n \in \No, \ \ell + n \geq 0.
$$
Since $(\rho_{n+\eta})_{n\in\No}$ is the maximum modulus sequence for $\sigma^{k+\eta}(\tilde{e})$ starting at $\rho_\eta$, we deduce that $x \in A_{\sigma^{k+\eta}(\tilde{e})}^{\ell}(f, \rho_\eta)$, and the claim follows.

Combining these results, we have shown that
$$
A_{\sigma^k(e)}^{\ell}(f, R') \supset A_{\sigma^{k}(\tilde{e})}^{\ell-\eta}(f, R).
$$

We deduce that
\begin{align*}
\bigcup_{\substack{\ell\in\Z \\ k\in\No}} A^{\ell}_{\sigma^k(e)}(f, R')
&\supset \bigcup_{\substack{\ell\in\Z \\ k\in\No}} A^{\ell-\eta}_{\sigma^k(\tilde{e})}(f, R) = \bigcup_{\substack{\ell\in\Z \\ k\in\No}} A^{\ell}_{\sigma^k(\tilde{e})}(f, R) \\
&\supset \bigcup_{\substack{\ell\in\Z \\ k\in\No}} A^{\ell}_{\sigma^{k+\eta}(\tilde{e})}(f, R) = \bigcup_{\substack{\ell\in\Z \\ k\in\No}} A^{\ell}_{\sigma^k(e)}(f, R),
\end{align*}
which completes the proof of \ref{aindepR}. \\

For part \ref{ainv}, we first prove that $A_e(f)$ is completely invariant.
Fix $R>R(f)$.
Suppose that $x \in A_e(f)$, in which case there exists $\ell\in\Z$ and $k \in \No$ such that $x \in A_{\sigma^k(e)}^{\ell}(f, R)$. It follows that $f(x) \in A_{\sigma^k(e)}^{\ell-1}(f, R) \subset A_e(f)$. The fact that $f(x) \in A_e(f)$ implies that $x \in A_e(f)$ follows very similarly, using Proposition~\ref{ldown}. Hence $A_e(f)$ is completely invariant. It then follows that the same it true for $A(f)$. \\

For part \ref{aequiv}, we first show that $A_e(f) = A_{\sigma(e)}(f)$.  For, taking $R>R(f)$,
$$
A_{\sigma(e)}(f) = \bigcup_{\substack{\ell\in\Z \\ k\in\No}} A_{\sigma^{k+1}(e)}^{\ell}(f, R) = \bigcup_{\substack{\ell\in\Z \\ k\in\N}} A_{\sigma^k(e)}^{\ell}(f, R) \subset A_e(f).
$$
Moreover, by Proposition~\ref{ldown},
$$
A_{e}(f) = \bigcup_{\substack{\ell\in\Z \\ k\in\No}} A_{\sigma^k(e)}^{\ell}(f, R) \subset \bigcup_{\substack{\ell\in\Z \\ k\in\No}} A_{\sigma^{k+1}(e)}^{\ell+1}(f, R) = \bigcup_{\substack{\ell\in\Z \\ k\in\No}} A_{\sigma^{k}(\sigma(e))}^{\ell}(f, R) = A_{\sigma(e)}(f).
$$
Hence $A_e(f) = A_{\sigma(e)}(f)$, as claimed. By repeated application of this equality, it then follows that $A_e(f) = A_{e'}(f)$ if $e$ and $e'$ are equivalent.

Next we note that if $R > \rho_S$ is sufficiently large, then the sets $$\{ x \in \widehat{\R^d}\setminus S : |x|_j > R\}, \qfor j \in \mathcal{P},$$ are pairwise disjoint. The fact that $A_e(f)$ and $A_{e'}(f)$ are disjoint if $e$ and $e'$ are not equivalent follows from this observation.
\end{proof}

 To show that each little fast escaping set is non-empty, we require a covering result that is analogous to \cite[Proposition 5.1]{MR3215194}.
\begin{lemma}
\label{lemma:cov}
Suppose that $f : \widehat{\R^d}\setminus S \to \widehat{\R^d} \setminus S$ is a quasiregular map of \mbox{$S$-transcendental} type, and  that $\alpha, \beta > 1$. Then there exists $r_0 > \rho_S$ such that, for all $r > r_0$ and all $j, k \in \mathcal{P}$, there exists $R > M_{j,k}(r,f)$ such that
\begin{equation}
\label{eq:cov}
f(\overline{A_j(r, \alpha r)}) \supset \overline{A_k(R, \beta R)}.
\end{equation}
\end{lemma}
\begin{proof}
Since $\mathcal{P}$ is a finite set, it suffices to show that there exists such an $r_0$ for any given $j, k \in \mathcal{P}$. Suppose then that $j, k \in \mathcal{P}$, and consider the quasiregular map $g = \phi_k \circ f \circ \phi_j^{-1}$, which is quasiregular on a punctured neighbourhood of $\infty$, and has an essential singularity at $\infty$. We are required to prove that there exists $r_0 > \rho_S$ such that for all $r > r_0$ there exists $R > M(r,g)$ such that $g(\overline{A(r, \alpha r)}) \supset \overline{A(R, \beta R)}$. The proof of this, using, for example, equation (\ref{anMequation}), is almost identical to that of \cite[Proposition 5.1]{MR3215194} and is omitted.
\end{proof}

\begin{proof}[Proof of Theorem~\ref{theo:A1}]
Suppose that $f : \widehat{\R^d}\setminus S \to \widehat{\R^d} \setminus S$ is a quasiregular map of \mbox{$S$-transcendental} type and that $e \in \mathcal{P}^{\No}$. Let $R > R(f)$, where $R(f)$ is the constant from Lemma~\ref{lemm:mmseq}, and let $(R_n)_{n\in\No}$ be the maximum modulus sequence for $e$ starting at $R$. These definitions are in place throughout the proof. \\

First we prove that $A_e(f)$ is non-empty. Let $(r_n)_{n\in\No}$ be an increasing sequence of real numbers greater than $R$, such that
\begin{equation}
\label{eq:inA1}
r_{n+1} > M_{e_n,e_{n+1}}(r_n,f), \qfor n\in\No,
\end{equation}
and
\begin{equation}
\label{eq:inA2}
f(\overline{A_{e_n}(r_n, 2r_n)}) \supset \overline{A_{e_{n+1}}(r_{n+1}, 2r_{n+1})}, \qfor n\in\No.
\end{equation}
The existence of such a sequence follows from Lemma~\ref{lemma:cov}, with $\alpha = \beta = 2$. Observe that it follows from (\ref{eq:inA1}) and from Lemma~\ref{lemma:gen}\ref{Mnormal} that
\begin{equation}
\label{reallyinA}
r_n \geq R_n, \qfor n \in \No.
\end{equation}

It follows from (\ref{eq:inA2}), using Lemma~\ref{lemm:exists}, that there is a point $\xi \in \overline{A_{e_0}(r_0, 2r_0)}$ with the property that $f^{n}(\xi) \in \overline{A_{e_n}(r_n, 2r_n)}$, for $n\in\N$. The fact that $\xi \in A_e(f)$ follows from (\ref{reallyinA}).  \\

The proof that all components of $A_e(f)$ are $S$-unbounded is as follows. 
By Theorem~\ref{theo:A2}\ref{aindepR} we can assume that $R$ is as large as we wish. In particular, we assume that $R=R_0$ is sufficiently large that the following holds for all $r > R_0$: If~$j,k \in \mathcal{P}$ and $x \in \R^d$ is such that $r/2 \leq |x|_j < r$, then
\begin{equation}
\label{smallmaxmod}
|f(x)|_k < M_{j, k}(r, f).
\end{equation}

By Proposition~\ref{ldown}, every point in $A_e(f)$ lies in a little level set $A_{\sigma^k(e)}^{\ell}(f, R)$ for some $\ell, k \in \No$. We aim to prove that every component of all such sets is $S$-unbounded. Since $e$ is arbitrary, we can assume that $k=0$. Next, we show that there is no loss of generality in assuming that $\ell=0$.  To see this, suppose that $y\in A_e^\ell(f,R)$ and that $f^{\ell}(y)$ lies in an $S$-unbounded component, $X$ say, of $A_e^0(f,R)$. Recalling that $A_e^0(f,R)$ is closed, it follows by Proposition~\ref{bflmpunctured} that the component of $f^{-\ell}(X)$ containing $y$ is $S$-unbounded. Hence, since $f^{-\ell}(X) \subset A_{e}^{\ell}(f, R)$, the point $y$ lies in an $S$-unbounded component of $A_{e}^{\ell}(f, R)$.


It remains to show that all components of
$$
A := A_e^{0}(f, R) = \{ x \in\widehat{\R^d}\setminus S: |f^{n}(x)|_{e_{n}} \geq R_n \text{ for } n\in\No \},
$$
are $S$-unbounded. We take $\xi \in A$ and now aim to show that $\xi$ lies in a connected subset of $A$ that is $S$-unbounded.

Fix a natural number $n$, and define sets
$$
X_{n,j} := f^{-j}(\{ x \in \widehat{\R^d} \setminus S : |x|_{e_n} \geq R_n \}), \qfor j \in \{0,1,\ldots, n\}.
$$
These sets are all closed (recall that topological operations are taken in $\widehat{\R^d} \setminus S$). Moreover, by Proposition~\ref{bflmpunctured}, all components of these sets are $S$-unbounded.

For each $j \in \{0,1,\ldots, n\}$, let $L_{n,j}$ denote the component of $X_{n,j}$ that contains $f^{n-j}(\xi)$; clearly these sets are also closed and connected. We claim that
\begin{equation}
\label{sizeclaim}
L_{n,j} \subset \{ x \in \widehat{\R^d} \setminus S : |x|_{e_{n-j}} \geq R_{n-j} \}, \qfor j \in \{0,1,\ldots, n\}.
\end{equation}

It is clear that (\ref{sizeclaim}) holds when $j=0$. Now suppose that (\ref{sizeclaim}) holds when $j=p-1$, for some $p \in \{1,\ldots,n\}$. Suppose, by way of contradiction, that there exists $w \in L_{n,p}$ such that $|w|_{e_{n-p}} < R_{n-p}$. Since $L_{n,p}$ is connected and contains $f^{n-p}(\xi)$, it follows that there exists $w' \in L_{n,p}$ such that $R_{n-p}/2 \leq |w'|_{e_{n-p}} < R_{n-p}$. Note that $f(w') \in L_{n,p-1}$, and so
$$
|f(w')|_{e_{n-p+1}} \geq R_{n-p+1} = M_{{e_{n-p}},e_{n-p+1}}(R_{n-p},f).
$$
We deduce a contradiction to (\ref{smallmaxmod}).

For simplicity we now set $L_n := L_{n,n}$. We have obtained that, for each $n\in\N$, $L_n$ is a closed, connected, $S$-unbounded set, lying in $\{ x \in \widehat{\R^d} \setminus S : |x|_{e_{0}} \geq R_{0} \}$. Moreover, since $L_n$ is $S$-unbounded and $\{ x \in \widehat{\R^d} : |x|_{e_{0}} \geq R_{0} \} \cap S = \{y_{e_0}\}$, by the choice of $R_0$, we have that $L_n \cup \{y_{e_0}\}$ is connected, and is a continuum.

From (\ref{sizeclaim}) we see that $L_{n+1,1}\subset X_{n,0}$, and it easily follows that 
$L_{n+1} \subset L_n$. We deduce that $(L_n \cup \{y_{e_0}\})_{n\in\N}$ is a nested sequence of continua, each of which contains $\xi$ and $y_{e_0}$. It then follows, by \cite[Theorem 1.8]{MR1192552}, that
$$
K = \bigcap_{n\in\N} (L_n \cup \{y_{e_0}\}),
$$
is a continuum containing $\xi$ and $y_{e_0}$. Observe also that  $K\setminus\{y_{e_0}\} \subset A$.

Finally, we let $X'$ be the component of $K \setminus \{y_{e_0}\}$ that contains $\xi$. Clearly $X'$ is connected, and it follows by Proposition~\ref{kurat}, with $X = K \setminus \{y_{e_0}\}$, that $X'$ is $S$-unbounded, as required.
\end{proof}

\begin{remark}\normalfont
It is natural to ask if $A(f) = A(f^p)$, holds for $p\in\N$; this is known to be true for quasiregular maps of $\R^d$ of transcendental type \cite[Proposition 3.1]{MR3215194}. In fact this is not the case, even for transcendental analytic self-maps of $\C^*$. For example, let $f$ be the transcendental analytic self-map of $\C^*$ given by $$f(z) := \exp(\exp(1/z) + z).$$ It can be shown that if $x>0$ is large, then $x \in A_e(f)$, where $e = 0 0 \ldots$. Hence $x \in A(f)$. However, $M_{0,0}(r,f^2) \geq \exp\exp(e^{r/2})$, for large values of $r$. It can then be deduced that $x \notin A(f^2)$, and so $A(f) \ne A(f^2)$.
\end{remark}

%
%
%
%
%
\section{The boundary of the fast escaping set}\label{section:AJ}
In this section we prove Theorem~\ref{theo:AJ}. We begin by considering the set $BO(f)$ of points whose forward orbit is $S$-bounded; in other words,
$$
BO(f) := \{ x \in \widehat{\R^d} \setminus S : \exists \ L > 0 \text{ s.t. } |f^k(x)|_j < L, \text{ for } k \in \No,\, j \in \mathcal{P}\}.
$$
\begin{proposition}
\label{prop:BO}
Suppose that $f : \widehat{\R^d}\setminus S \to \widehat{\R^d} \setminus S$ is a quasiregular map of \mbox{$S$-transcendental} type. Then $BO(f)$ is infinite.
\end{proposition}
\begin{proof}
Let $q$ be Rickman's constant. By taking $N=q+1$ and applying Proposition~\ref{danprop2} as in the proof of Proposition~\ref{prop:capclaim}, we see that for any sufficiently large $s> \rho_S$ the sets $U_i=A(3^{i}s, 2\cdot 3^{i}s)$ have the property that, for each $j\in\{1,\ldots,N\}$, there exists $i\in\{1,\ldots,N\}$ such that $f(U_j)\supset U_i$.
It follows that there exists ${j'} \in \{1, \ldots, N\}$ and $p \in \N$ such that $f^p(U_{j'}) \supset U_{j'}$. Since $s > \rho_S$, the sets ${U_i}$ 
are all $S$-bounded. We can deduce, using Lemma~\ref{lemm:exists}, that there exists $\xi \in \overline{U_{j'}}$ such that $f^{kp}(\xi) \in \overline{U_{j'}}$, for $k\in\No$. Hence $\xi \in BO(f)$. Since we are free to choose $s$ arbitrarily large, we deduce that $BO(f)$ is infinite.
\end{proof}

We require the following, which is extracted from the proof of \cite[Theorem 3.3]{MR686641}. If $a \in \R^d$ and $r > 0$, then we let $B(a,r) := \{ x \in \R^d : |x-a| < r \}$.
\begin{lemma}
\label{RVTheorem3.3}
Suppose that $G \subset \R^d$ is an unbounded domain with the property that there exist $\delta>0$ and $r_0 > 0$ such that
\begin{equation}
\label{RVcondition}
\operatorname{cap}(B(0,2r), (\R^d\setminus G) \cap \overline{B(0,r)}) \geq \delta, \qfor r \geq r_0.
\end{equation}
Suppose also that $g : G \to \R^d$ is a non-constant quasiregular map such that ${|g(x_0)|>1}$, for some $x_0\in G$, and
\begin{equation}
\label{preveq}
\limsup_{x \rightarrow y} |g(x)| \leq 1, \qfor y \in \partial G,
\end{equation}
where the boundary in (\ref{preveq}) is taken in $\R^d$. Then there exist $\alpha > 1$, $\beta \in (0,1)$ and $\rho_0>0$ such that if $r/\alpha>s>\rho_0$, then
\[ \log\log M_G(s,g) \le \left(\frac{\log(r/s)}{\log\alpha} - 1\right)\log(1-\beta) + \log\log M_G(r,g), \]
where $M_G(r,g)= \sup\{|g(x)| : x \in G, \ |x| = r\}$.
\end{lemma}

We observe that the condition (\ref{RVcondition}) certainly holds if $\R^d \setminus G$ contains an unbounded curve; see, for example, \cite[Remark 3.4]{MR686641}. We use Lemma~\ref{RVTheorem3.3} to prove the following.
\begin{lemma}
\label{lemma:grows}
Suppose that $f : \widehat{\R^d}\setminus S \to \widehat{\R^d} \setminus S$ is a quasiregular map of \mbox{$S$-transcendental} type. Then there exist $\rho_0>\rho_S$, $\alpha>1$ and $B>0$ such that if $r/\alpha > s > \rho_0$ and $j, k \in \mathcal{P}$, then
\begin{equation}
\label{Mgrows}
\log \frac{\log M_{j,k}(r,f)}{\log M_{j,k}(s,f)} \ge B\log \frac{r}{s}.
\end{equation}
\end{lemma}
\begin{proof}
Since $\mathcal{P}$ is a finite set, it will suffice to prove the result for some fixed choice of $j,k \in\mathcal{P}$. By \eqref{rhoSdef}, the map $g := \phi_k \circ f \circ \phi_j^{-1}$ is quasiregular on the set $\{x\in\R^d : |x| \ge \rho_S\}$. It follows from Lemma~\ref{lemma:asympt}, together with the fact that $\nu > 0$, that there is a curve which accumulates only at infinity, $\Gamma$ say, on which $g$ is bounded. By taking a subcurve of $\Gamma$, if necessary, we can assume that $\Gamma$ is simple, and that $\Gamma \subset \{ x \in \R^d : |x| > \rho_S \}$. It follows that $$G := \{ x \in \R^d : |x| > \rho_S \} \setminus \Gamma$$
is a domain. Note that $g$ is unbounded on $G$ because $g$ has an essential singularity at infinity.
 We deduce by Lemma~\ref{RVcondition}, applied to a constant multiple of $g$, that there exist $\rho_0>\rho_S$, $\alpha>1$ and $B,C>0$ such that, for $r/\alpha > s > \rho_0$,
\[  \log \frac{\log M_{j,k}(r,f)}{\log M_{j,k}(s,f)} = \log \frac{\log M_G(r,g)}{\log M_G(s,g)} \ge 2B\log\left(\frac{r}{s}\right) - C.   \]
The result follows, by increasing $\alpha$ if necessary.
\end{proof}
Next, we apply Lemma~\ref{lemma:grows} to compare the rate of growth of two related maximum modulus sequences.
\begin{lemma}\label{lemma:En}
Suppose that $f : \widehat{\R^d}\setminus S \to \widehat{\R^d} \setminus S$ is a quasiregular map of \mbox{$S$-transcendental} type. Let $\rho_0$, $\alpha$ and $B$ be as given by Lemma~\ref{lemma:grows}. Suppose that $R>R(f)$ satisfies
\begin{equation}\label{Rdef}
R > \max\{\rho_0, \alpha, \exp(2/B) \} \quad \mbox{and} \quad \frac{\log M_{j,k}(R,f)}{\log R} \ge \exp(1), \qfor j,k\in\mathcal{P}.
\end{equation}
Let $e \in \mathcal{P}^\No$ and let $(R_n)_{n\in\No}$ be the maximum modulus sequence for $e$ starting at~$R$ and $(S_n)_{n\in\No}$ be the maximum modulus sequence for $\sigma(e)$ starting at $R$. Then
\begin{equation}\label{ineq:En}
  \log \frac{\log R_n}{\log S_{n-1}} \ge E_n,
\end{equation}
for $n\ge 1$, where $(E_n)$ denotes the iterated exponential sequence given by $E_1=1$ and $E_n=\exp(E_{n-1})$ for $n\ge2$.
\end{lemma}
Note that by Lemma \ref{lemma:gen}\ref{Mbig}, all sufficiently large $R$ satisfy condition \eqref{Rdef} above.
\begin{proof}[Proof of Lemma~\ref{lemma:En}]
The $n=1$ case holds because
\[ \log \frac{\log R_1}{\log S_0} = \log \frac{\log M_{ e_0, e_{1}}(R,f)}{\log R} \ge 1 = E_1, \]
by \eqref{Rdef}. So we assume that $n\ge 2$ and that
\[ \log \frac{\log R_{n-1}}{\log S_{n-2}} \ge E_{n-1}, \]
and we aim to prove \eqref{ineq:En}. Our assumption implies that $$R_{n-1}/S_{n-2} \ge {S_{n-2}}^{E_n-1} \ge R^{E_n-1} > \alpha.$$ Using \eqref{Rdef} and Lemma~\ref{lemma:grows} now leads to
\begin{align*}
  \log \frac{\log R_n}{\log S_{n-1}} &=  \log \frac{\log M_{ e_{n-1}, e_{n}}(R_{n-1},f)}{\log M_{ e_{n-1}, e_{n}}(S_{n-2},f)} \\
  &\ge B\log \frac{R_{n-1}}{S_{n-2}} \\
  &\ge B(E_n-1)\log R \\
  &\ge 2(E_n-1) \ge E_n. \qedhere
\end{align*}
\end{proof}
%
%
\begin{proof}[Proof of Theorem~\ref{theo:AJ}]
Suppose that $f : \widehat{\R^d}\setminus S \to \widehat{\R^d} \setminus S$ is a quasiregular map of \mbox{$S$-transcendental} type. Let $e \in \mathcal{P}^\No$. Suppose first that $x \in J(f)$, and let $U\subset \widehat{\R^d}\setminus S$ be a neighbourhood of $x$. It follows from Proposition~\ref{prop:BO}, and the definition of the Julia set, that there exists $\xi_0 \in U$ and $n_0 \in \N$ such that $f^{n_0}(\xi_0) \in BO(f)$. We deduce that $U$ meets the complement of $A_e(f)$. By Theorem~\ref{theo:A1}, $A_e(f)$ contains an $S$-unbounded component, and so is infinite. Hence, we can deduce in the same way that there exists $\xi_1 \in U$ and $n_1 \in \N$ such that $f^{n_1}(\xi_1) \in A_e(f)$. We deduce, by Theorem~\ref{theo:A2}\ref{ainv}, that $U$ meets $A_e(f)$. Since $U$ was arbitrary, we conclude that $J(f) \subset \partial A_e(f)$. It follows by a similar argument that $J(f) \subset \partial A(f)$.

We next show that $\partial A_e(f) \subset J(f)$, for each $e \in\mathcal{P}^{\No}$. The proof of this fact is similar to the proof of \cite[Theorem 1.2]{MR3265357}, although we give sufficient detail to show how that proof transfers into our setting. We note that a key difference between \cite[Theorem 1.2]{MR3265357} and Theorem~\ref{theo:AJ} is that, in our setting, a result on the growth of the generalised maximum modulus functions (equation (\ref{Mgrows})) always holds, whereas for the transcendental type quasiregular self-maps of $\R^d$ studied in \cite{MR3265357}, an analogous property needs to be taken as an additional hypothesis.


We now let $e \in \mathcal{P}^{\No}$, and suppose, by way of contradiction, that there exists a point $x_0 \in \partial A_e(f)$ such that $x_0 \notin J(f)$. Since $J(f)$ is closed, we can let $r>0$ be sufficiently small that $$B(x_0, 4r) \cap (J(f) \cup S) = \emptyset.$$ We can also assume, since $J(f) = J_{\operatorname{cap}}(f)$ (see Proposition~\ref{J=J_cap}), that $r$ is sufficiently small that the iterates of $f$ omit a set of positive capacity in $B(x_0, 4r)$.

Since $x_0 \in \partial A_e(f)$, we can choose points
$$x_A \in B(x_0,r) \cap A_e(f) \quad\text{and}\quad x_N \in B(x_0,r) \setminus A_e(f).$$
Take $R>R(f)$ sufficiently large that condition \eqref{Rdef} is satisfied.
From Theorem~\ref{theo:A2}\ref{aindepR} and Proposition~\ref{ldown}, we see that there exist $\ell \in \No$ and $k \in \No$ such that $x_A\in A_{\sigma^k(e)}^\ell(f,R)$. Let $(R_n)_{n\in\No}$ be the maximum modulus sequence for $\sigma^k(e)$ starting at~$R$, and let $(S_n)_{n\in\No}$ be the maximum modulus sequence for $\sigma^{k+1}(e)$ starting at~$R$.
Lemma~\ref{lemma:En} tells us that $R_n$ grows much faster than $S_{n-1}$; our aim now is to seek a contradiction to this fact.

By the definition of $A_{\sigma^k(e)}^\ell(f,R)$, we have that
\[ |f^{\ell+n}(x_A)|_{e_{k+n}} \geq R_n, \qfor n\in\No. \]
We claim that, on the other hand, the fact that $x_N\notin A_e(f)$ implies that there exists an infinite set $\mathcal{N} \subset \N$ such that
\begin{equation}
\label{newclaim}
|f^{\ell+n}(x_N)|_{e_{k+n}} < S_{n-1}, \qfor n \in \mathcal{N}.
\end{equation}
For, suppose that (\ref{newclaim}) does not hold, in which case there exists $n_0 \in \N$ such that
\begin{equation}
\label{newclaim2}
|f^{\ell+n}(x_N)|_{e_{k+n}} \geq S_{n-1}, \qfor n \geq n_0.
\end{equation}
Set $\ell' = \ell + n_0$, $k' = k + n_0$ and $S'_n = S_{n_0-1+n}$, for $n \in \No$. It then follows from (\ref{newclaim2}) that
$$
|f^{\ell'+n}(x_N)|_{e_{k'+n}} \geq S'_{n}, \qfor n \geq 0.
$$
Since $(S_n)_{n\in\No}$ is the maximum modulus sequence for $\sigma^{k+1}(e)$ starting at $R$, it follows that $(S'_n)_{n\in\No}$ is the maximum modulus sequence for $\sigma^{k'}(e)$ starting at $S'_0$. It follows that $x_N \in A^{\ell'}_{\sigma^{k'}(e)}(f, S'_0)$ and so $x_N \in A_e(f)$. This contradiction completes the proof of our claim (\ref{newclaim}).

Next, for $n\in\mathcal{N}$, we set
$$
X_{A,n} :=  \{ x \in B(x_0, 2r) : |f^{\ell+n}(x)|_{e_{k+n}} \geq R_{n} \}
$$
and
$$
X_{N,n} := \{ x \in B(x_0, 2r) : |f^{\ell+n}(x)|_{e_{k+n}} \leq S_{n-1} \}.
$$
For $I \in \{A, N\}$ and $n\in\mathcal{N}$, we denote the component of $X_{I,n}$ that contains $x_I$ by $Y_{I,n}$. We assert that the closures of $Y_{A,n}$ and $Y_{N,n}$ both meet $\partial B(x_0, 2r)$. In fact, Proposition~\ref{bflmpunctured} tells us that all components of $(f^{\ell+n})^{-1}(E)$ are $S$-unbounded when $E$ is either $\{y\in \widehat{\R^d} : |y|_{e_{k+n}}\ge R_n\}$ or $\{y\in \widehat{\R^d} : |y|_{e_{k+n}}\le S_{n-1}\}$, and this implies our assertion.
 In particular, the connected sets $Y_{A,n}$ and $Y_{N,n}$ both have diameter at least $r$. 



Now, recalling notation from Section~\ref{sect:modulus}, let $\Gamma_n := \Delta(Y_{A,n}, Y_{N,n}; B(x_0, 2r))$. Following \cite{MR3265357}, we note that by \cite[Lemma~5.42]{MR950174}, there exists $\epsilon = \epsilon(d) > 0$ such that \begin{equation}
\label{Mlower}
M(\Gamma_n) \geq \epsilon, \qfor n\in\mathcal{N}.
\end{equation}

On the other hand, suppose that $n \in\mathcal{N}$, and let $F$ be the quasiregular map $F := \phi_{e_{k+n}} \circ f^{\ell+n}$.
For exactly the same reasons as in \cite{MR3265357}, it can be deduced that there is a constant $L = L(d) > 0$ such that
\begin{equation}\label{Mupper}
M(\Gamma_n) \leq L K_O(F) K_I(F) \left(\log\frac{\log R_n}{\log S_{n-1}}\right)^{1-d}.
\end{equation}
Now, by (\ref{Keq}), we have $K_I(F) \leq K_I(f)^{\ell+n}$ and $K_O(F) \leq K_O(f)^{\ell+n}$. Hence, with $K = (K_O(f)K_I(f))^{1/(d-1)}$ we obtain, using (\ref{Mlower}) and (\ref{Mupper}), that there is a constant $L' = L'(d,f,\ell) > 0$ such that
$$
\log\frac{\log R_{n}}{\log S_{n-1}} \leq L' K^n.
$$

This is in contradiction to \eqref{ineq:En} for large $n\in\mathcal{N}$, and so $\partial A_e(f) \subset J(f)$ as required. \\

It remains to show that $J(f) = \partial A(f)$. Since we have that $J(f) \subset \partial A(f)$, we need to show that $\partial A(f) \subset J(f)$. Suppose, by way of contradiction, that there exists $x \in \partial A(f) \cap QF(f)$. Let $U \subset QF(f)$ be a neighbourhood of $x$. Then $U$ meets $A(f)$, and so there exists $e \in \mathcal{P}^\No$ such that $U$ meets $A_e(f)$. Since $\partial A_e(f) = J(f)$, it follows that $U \subset A_e(f)$, and so $U \subset A(f)$. This contradiction completes the proof.
\end{proof}
%
%
%
%
%
%
\section{Proof of the final part of Theorem~\ref{theo:J}}
\label{section:julia2}
Suppose that $f : \widehat{\R^d}\setminus S \to \widehat{\R^d} \setminus S$ is a quasiregular map of \mbox{$S$-transcendental} type. In this section we prove Theorem~\ref{theo:J}\ref{jnoSbounded}, which states that $J(f)$ has no $S$-bounded components.

Suppose, by way of contradiction, that $K$ is an $S$-bounded component of $J(f)$. Since $J(f)$ is closed, it follows by Proposition~\ref{prop:comps}, with $X = J(f)$, that there is an $S$-bounded domain $V$ such that $K \subset V$ and $\partial V \subset QF(f)$.

By Corollary~\ref{corr:J} there exists $e \in \mathcal{P}^{\No}$  such that $A_e(f)\subset J(f)$.
Each point of $K$ is in $J(f)$, and so, by Theorem~\ref{theo:AJ}, is also in $\partial A_e(f)$. It follows that $V$ meets $A_e(f)$. Hence $\partial V$ also meets $A_e(f)$, because, by Theorem~\ref{theo:A1}, every component of $A_e(f)$ is $S$-unbounded. However, $\partial V$ lies in $QF(f)$ and $A_e(f)$ does not meet $QF(f)$. This contradiction completes the proof.
%
%
%
%
\section{Proof of Theorem~\ref{theo:F}}
\label{section:fatou}
Suppose that $f : \widehat{\R^d}\setminus S \to \widehat{\R^d} \setminus S$ is a quasiregular map of \mbox{$S$-transcendental} type. We first show that all quasi-Fatou components of $f$ are $S$-full. Suppose, by way of contradiction, that $U$ is a quasi-Fatou component of $f$ that is $S$-hollow. Let $X$ be an $S$-bounded complementary component of $U$. Clearly $X$ meets $J(f)$ and so 
$X$ contains a component of $J(f)$. Hence this component of $J(f)$ is $S$-bounded, in contradiction to Theorem~\ref{theo:J}\ref{jnoSbounded}.

In the remainder of the proof, for simplicity, complementary components are taken in $\widehat{\R^d}$. We need to show that the number of quasi-Fatou components of $f$ having more than one complementary component is at most $\nu$. Suppose, by way of contradiction, that there is a set $V$ that is the union of $\nu+1$ quasi-Fatou components, each of which has more than one complementary component. It can be shown by an inductive argument that $V$ has at least $\nu+2$ complementary components.

Since $S$ has $\nu+1$ elements, it follows that $V$ has a complementary component, $X$ say, which does not meet $S$. Then $X$ certainly meets $J(f)$ and so 
$X$ contains a component of $J(f)$. Hence this component of $J(f)$ is $S$-bounded, in contradiction to Theorem~\ref{theo:J}\ref{jnoSbounded}. \\

%
%
%
%
%
\noindent \emph{Acknowledgment:} The authors are grateful to David Mart{\'{\i}}-Pete for many helpful discussions.
%
%
%
%
%
%

\begin{thebibliography}{10}

\bibitem{MR951969}
{\sc Baker, I.~N.}
\newblock Wandering domains for maps of the punctured plane.
\newblock {\em Ann. Acad. Sci. Fenn. Ser. A I Math. 12}, 2 (1987), 191--198.

\bibitem{MR1687848}
{\sc Baker, I.~N., and Dom{\'{\i}}nguez, P.}
\newblock Analytic self-maps of the punctured plane.
\newblock {\em Complex Variables Theory Appl. 37}, 1-4 (1998), 67--91.

\bibitem{beardon}
{\sc Beardon, A.}
\newblock {\em {Iteration of rational functions}}.
\newblock Springer, 1991.

\bibitem{MR1216719}
{\sc Bergweiler, W.}
\newblock Iteration of meromorphic functions.
\newblock {\em Bull. Amer. Math. Soc. (N.S.) 29}, 2 (1993), 151--188.

\bibitem{MR1357963}
{\sc Bergweiler, W.}
\newblock On the {J}ulia set of analytic self-maps of the punctured plane.
\newblock {\em Analysis 15}, 3 (1995), 251--256.

\bibitem{MR2248829}
{\sc Bergweiler, W.}
\newblock Fixed points of composite entire and quasiregular maps.
\newblock {\em Ann. Acad. Sci. Fenn. Math. 31}, 2 (2006), 523--540.

\bibitem{MR3009101}
{\sc Bergweiler, W.}
\newblock Fatou-{J}ulia theory for non-uniformly quasiregular maps.
\newblock {\em Ergodic Theory Dynam. Systems 33}, 1 (2013), 1--23.

\bibitem{MR3215194}
{\sc Bergweiler, W., Drasin, D., and Fletcher, A.}
\newblock The fast escaping set for quasiregular mappings.
\newblock {\em Anal. Math. Phys. 4}, 1-2 (2014), 83--98.

\bibitem{MR2448586}
{\sc Bergweiler, W., Fletcher, A., Langley, J., and Meyer, J.}
\newblock The escaping set of a quasiregular mapping.
\newblock {\em Proc. Amer. Math. Soc. 137}, 2 (2009), 641--651.

\bibitem{MR3265357}
{\sc Bergweiler, W., Fletcher, A., and Nicks, D.~A.}
\newblock The {J}ulia set and the fast escaping set of a quasiregular mapping.
\newblock {\em Comput. Methods Funct. Theory 14}, 2-3 (2014), 209--218.

\bibitem{MR1684251}
{\sc Bergweiler, W., and Hinkkanen, A.}
\newblock On semiconjugation of entire functions.
\newblock {\em Math. Proc. Cambridge Philos. Soc. 126}, 3 (1999), 565--574.

\bibitem{MR3265283}
{\sc Bergweiler, W., and Nicks, D.~A.}
\newblock Foundations for an iteration theory of entire quasiregular maps.
\newblock {\em Israel J. Math. 201}, 1 (2014), 147--184.

\bibitem{Bhattacharyya}
{\sc Bhattacharyya, P.}
\newblock {\em Iteration of analytic functions}.
\newblock PhD thesis, University of London, 1969.

\bibitem{DrasinandPankka}
{\sc Drasin, D., and Pankka, P.}
\newblock Sharpness of {R}ickman's {P}icard theorem in all dimensions.
\newblock {\em Acta Mathematica 214}, 2 (2015), 209--306.

\bibitem{MR1279126}
{\sc Hinkkanen, A.}
\newblock Completely invariant components in the punctured plane.
\newblock {\em New Zealand J. Math. 23}, 1 (1994), 65--69.

\bibitem{MR955806}
{\sc Keen, L.}
\newblock Dynamics of holomorphic self-maps of {${\mathbb{C}}^*$}.
\newblock In {\em Holomorphic functions and moduli, {V}ol.\ {I} ({B}erkeley,
  {CA}, 1986)}, vol.~10 of {\em Math. Sci. Res. Inst. Publ.} Springer, New
  York, 1988, pp.~9--30.

\bibitem{MR1120038}
{\sc Kotus, J.}
\newblock Iterated holomorphic maps on the punctured plane.
\newblock In {\em Dynamical systems ({S}opron, 1985)}, vol.~287 of {\em Lecture
  Notes in Econom. and Math. Systems}. Springer, Berlin, 1987, pp.~10--28.

\bibitem{MR0259835}
{\sc Kuratowski, K.}
\newblock {\em Topology. {V}ol. {II}}.
\newblock New edition, revised and augmented. Translated from the French by A.
  Kirkor. Academic Press, New York-London; Pa\'nstwowe Wydawnictwo Naukowe
  Polish Scientific Publishers, Warsaw, 1968.

\bibitem{MR916928}
{\sc Makienko, P.~M.}
\newblock Iterations of analytic functions in {${\bf C}^*$}.
\newblock {\em Dokl. Akad. Nauk SSSR 297}, 1 (1987), 35--37, translation in Soviet Math. Dokl. 36 (1988), no. 3, 418--420.

\bibitem{DMP1}
{\sc {Mart{\'{\i}}-Pete}, D.}
\newblock {The escaping set of transcendental self-maps of the punctured
  plane}.
\newblock {\em Ergodic Theory and Dynamical Systems, FirstView} (2016), doi:10.1017/etds.2016.36.

\bibitem{MR633273}
{\sc Miniowitz, R.}
\newblock Normal families of quasimeromorphic mappings.
\newblock {\em Proc. Amer. Math. Soc. 84}, 1 (1982), 35--43.

\bibitem{MR1192552}
{\sc Nadler, Jr., S.~B.}
\newblock {\em Continuum theory}, vol.~158 of {\em Monographs and Textbooks in
  Pure and Applied Mathematics}.
\newblock Marcel Dekker, Inc., New York, 1992.


\bibitem{MR0132534}
{\sc Newman, M.~H.~A.}
\newblock {\em Elements of the topology of plane sets of points}, second edition.
\newblock Cambridge University Press, New York, 1961.

\bibitem{danslow}
{\sc Nicks, D.~A.}
\newblock {Slow escaping points of quasiregular mappings}.
\newblock {\em Math. Z. 284},  (2016), 1053--1071.

\bibitem{SixsmithNicks1}
{\sc Nicks, D.~A., and Sixsmith, D.}
\newblock Hollow quasi-{F}atou components of quasiregular maps.
\newblock {\em Math. Proc. Cambridge Philos. Soc., FirstView} (2016), doi:10.1017/S0305004116000840.

\bibitem{SixsmithNicks2}
{\sc Nicks, D.~A., and Sixsmith, D.}
\newblock Periodic domains of quasiregular maps.
\newblock To appear in {\em Ergodic Theory Dynam. Systems}. Preprint, arXiv:1509.06723 (2015).

\bibitem{SixsmithNicks3}
{\sc Nicks, D.~A., and Sixsmith, D.}
\newblock The size and topology of quasi-{F}atou components of quasiregular
  maps.
\newblock {\em Proc. Amer. Math. Soc. 145}, (2017), 749--763.

\bibitem{MR3220456}
{\sc Okuyama, Y., and Pankka, P.}
\newblock Accumulation of periodic points for local uniformly quasiregular mappings.
\newblock {\em RIMS K\^{o}ky\^{u}roku Bessatsu B43\/} (2013), 121--139.

\bibitem{MR0056702}
{\sc R{\aa}dstr{\"o}m, H.}
\newblock On the iteration of analytic functions.
\newblock {\em Math. Scand. 1\/} (1953), 85--92.

\bibitem{MR583633}
{\sc Rickman, S.}
\newblock On the number of omitted values of entire quasiregular mappings.
\newblock {\em J. Analyse Math. 37\/} (1980), 100--117.

\bibitem{MR1238941}
{\sc Rickman, S.}
\newblock {\em Quasiregular mappings}, vol.~26 of {\em Ergebnisse der
  Mathematik und ihrer Grenzgebiete (3)}.
\newblock Springer-Verlag, Berlin, 1993.

\bibitem{MR686641}
{\sc Rickman, S., and Vuorinen, M.}
\newblock On the order of quasiregular mappings.
\newblock {\em Ann. Acad. Sci. Fenn. Ser. A I Math. 7}, 2 (1982), 221--231.

\bibitem{MR2792984}
{\sc Rippon, P.~J., and Stallard, G.~M.}
\newblock Slow escaping points of meromorphic functions.
\newblock {\em Trans. Amer. Math. Soc. 363}, 8 (2011), 4171--4201.

\bibitem{Rippon01102012}
{\sc Rippon, P.~J., and Stallard, G.~M.}
\newblock Fast escaping points of entire functions.
\newblock {\em Proc. London Math. Soc. (3) 105}, 4 (2012), 787--820.

\bibitem{MR950174}
{\sc Vuorinen, M.}
\newblock {\em Conformal geometry and quasiregular mappings}, vol.~1319 of {\em
  Lecture Notes in Mathematics}.
\newblock Springer-Verlag, Berlin, 1988.

\end{thebibliography}

\end{document}